\newif\ifJOURNAL\global\JOURNALfalse
\newif\ifHYPER\global\HYPERtrue
\newif\ifINTERNAL\global\INTERNALfalse
\ifJOURNAL
\documentclass[USenglish,onecolumn]{article}
\usepackage[big]{dgruyter}
\publisherlogo{dg-degruyter}
\else
\documentclass[10pt,b5paper,fleqn]{article}
\usepackage[utf8]{inputenc}
\usepackage{amsmath,amsfonts,amssymb}
\usepackage[mathscr]{eucal}
\usepackage{color}
\usepackage{graphicx}
\usepackage{amsthm}
\fi

\ifHYPER
\usepackage{hyperref}
\ifJOURNAL
\else
\definecolor{ks-green}{rgb}{0.0,0.7,0.0}
\definecolor{ks-red}{rgb}{0.7,0.0,0.0}
\definecolor{ks-blue}{rgb}{0.0,0.0,0.7}
\hypersetup{colorlinks=true,citecolor=ks-green,linkcolor=ks-red}
\fi
\fi

\numberwithin{equation}{section}
\theoremstyle{plain}
\newtheorem{theorem}[equation]{Theorem}
\newtheorem{lemma}[equation]{Lemma}
\newtheorem{corollary}[equation]{Corollary}
\newtheorem{proposition}[equation]{Proposition}
\newtheorem{algorithm}[equation]{Algorithm}
\theoremstyle{definition}
\newtheorem{definition}[equation]{Definition}
\newtheorem{Example}[equation]{Example}
\newtheorem{Remark}[equation]{Remark}

\newenvironment{remark}{\emph{Remark.}}{}
\newenvironment{remarks}{\emph{Remarks.}}{}
\newenvironment{notation}{\emph{Notation.}}{}

\ifJOURNAL
\newcommand{\lnum}[1]{\makebox[0.5cm][r]{\textnormal{\footnotesize{#1}}}}
\newenvironment{algtest}{\bgroup\bfseries\upshape
  \tabbing
    \hspace*{0.7cm}\=
    \hspace*{1.5em}\=\hspace*{1.5em}\=%
    \hspace*{1.5em}\=\hspace*{1.5em}\= \kill \>}{%
  \endtabbing\egroup}
\else
\newcommand{\lnum}[1]{\makebox[0.5\mathindent][r]{\textnormal{\footnotesize{#1}}}}
\newenvironment{algtest}{\bgroup\bfseries\upshape
  \tabbing
    \hspace*{\mathindent}\=
    \hspace*{1.5em}\=\hspace*{1.5em}\=%
    \hspace*{1.5em}\=\hspace*{1.5em}\= \kill \>}{%
  \endtabbing\egroup}
\fi

\DeclareMathOperator{\lgcd}{\mathrm{lgcd}}
\newcommand{\block}[1]{\underline{#1}}

\newcommand{\tabstrut}{\rule[-0.5ex]{0pt}{2.8ex}}

\newcommand{\mthstrut}{\rule[-0.5ex]{0pt}{2.2ex}}
\newcommand\trp{^{\!\top}}
\newcommand\inv{^{-1}}
\newcommand\numN{\mathbb{N}}

\newcommand{\freeALG}[2]{#1\langle #2\rangle}
\newcommand{\freeFLD}[2]{#1(\!\langle #2\rangle\!)}
\newcommand{\perm}{\Sigma}
\newcommand{\field}[1]{\mathbb{#1}}
\newcommand{\als}[1]{\mathcal{#1}}
\newcommand{\complexity}{\mathcal{O}}
\newcommand{\aclo}[1]{\overline{#1}}

\DeclareMathOperator{\rank}{rank}

\DeclareMathOperator{\pivot}{\#_\text{pb}}
\DeclareMathOperator{\linsp}{span}

\begin{document}
\ifJOURNAL
\else
\title{A Standard Form in (some) Free Fields:\\
  How to construct Minimal Linear Representations}
\author{Konrad Schrempf%
  \footnote{Contact: math@versibilitas.at (Konrad Schrempf),
    \url{https://orcid.org/0000-0001-8509-009X},
    Universität Wien, Fakultät für Mathematik,
    Oskar-Morgenstern-Platz~1, 1090 Wien, Austria.
    }
  \hspace{0.2em}\href{https://orcid.org/0000-0001-8509-009X}{%
  \includegraphics[height=10pt]{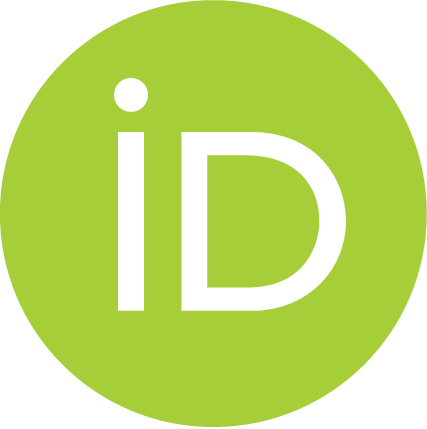}}
  }
\fi

\ifJOURNAL
\else
\maketitle
\fi

\begin{abstract}
{We describe a standard form for the elements in
the universal field of fractions of free associative
algebras (over a \emph{commutative} field).
It is a special version of the normal form
provided by Cohn and Reutenauer
and enables the use of linear algebra techniques for the
construction of minimal linear representations (in standard form)
for the sum and the product of two elements (given in standard form).
This completes ``minimal'' arithmetic in free fields
since ``minimal'' constructions for the inverse
are already known.
The applications are wide: linear algebra (over the
free field), rational identities, computing the
left gcd of two non-commutative polynomials, etc.}
\end{abstract}

\ifJOURNAL
\else
\medskip
\emph{Keywords and 2020 Mathematics Subject Classification.}
Free associative algebra,
minimal linear representation, admissible linear system,
left greatest common divisor, non-commutative polynomials;
Primary 16K40, 16Z05; Secondary 16G99, 16S85, 15A22
\fi

\section*{Introduction}

While the embedding of the integers into the field of
rational numbers is easy, that of non-commutative rings
(into skew fields) is not that straight forward even in special
cases \cite{Ore1931a}
. After Ore's construction it took almost forty years
and many contributors to develop a more general theory
\cite[Chapter~7]{Cohn2006a}
. The embedding of the \emph{free associative algebra}
(over a \emph{commutative} field) into a ring of quotients
(non-commutative localization) is even classified as
``ugly'' \cite{Lam1999a}
.

On the other hand there are many parallels between the
ring of integers and the free associative algebra,
for example, both have a \emph{distributive factor lattice}
(DFL) \cite{Cohn1982e}
\ or
\cite[Section~3.5]{Cohn1985a}
. And, starting from the \emph{normal form} (minimal linear representation)
of Cohn and Reutenauer
\cite{Cohn1994a}
\ of an element in the \emph{universal field of fractions}
(of the free associative algebra),
we will formulate a \emph{standard form} which can
be seen as a ``generalized'' fraction.
As a reminiscence to the work with ``classical''
fractions (we learn at school) we call them
briefly ``free fractions'' 
\cite{Schrempf2018c2}
.

For an introduction to free fields we recommend
\cite[Section~9.3]{Cohn2003b}
\ and with respect to linear representations
in particular \cite{Cohn1999a}
. For details we refer to 
\cite[Chapter~7]{Cohn2006a}
\ or 
\cite[Section~6.4]{Cohn1995a}
.
Since this work is only one part in a series,
further information and references can be found
in \cite{Schrempf2017a9}
\ (linear word problem, minimal inverse),
\cite{Schrempf2017b9}
\ (polynomial factorization) and
\cite{Schrempf2017c9}
\ (general factorization theory).

\medskip
The main idea of the \emph{standard form} is simple:
instead of viewing the system matrix of a linear representation
as a single ``block'' we transform it into a
block upper triangular form with smaller diagonal blocks.
If these \emph{pivot blocks} are small enough (called ``refined''),
\emph{linear}
techniques can be used to eliminate \emph{all} superfluous block rows
or columns, that is, solving ``local'' word problems,
in the ``sum'' or the ``product'' of two linear
representations, eventually yielding a \emph{minimal}
linear representation.
This can be accomplished with complexity $\complexity(dn^5)$
for an alphabet with $d$ letters and a linear representation
of dimension~$n$ with pivot blocks of size $\sqrt{n}$.
For the refinement of pivot blocks we need to solve
---at least in general--- systems of polynomial equations.
However \emph{linear} techniques can be used in some cases
to ``break'' big pivot blocks into smaller ones
(see Remark~\ref{rem:mr.linref}).

Since almost all here is rather elementary it should be noted
that it needs some effort to dig deep enough into the theory
in the background (in particular that of Cohn) to really
understand what is going on. Despite of the main results
(mentioned in the following) there is only one non-trivial
observation (formulated in Theorem~\ref{thr:mr.minchar}):
A given \emph{refined} linear representation 
is minimal if \emph{none} of the \emph{linear} systems of equations
for block row or column minimization has a solution.
(In the general non-refined case nothing can be said about minimality if there
is no more ``linear'' minimization step possible.)

In other words: while the \emph{normal form} 
\cite{Cohn1994a}
\ ``linearizes'' the word problem in free fields 
\cite[Section~2]{Schrempf2017a9}
, the \emph{standard form}
``linearizes'' (the minimization of) the sum and the
product (of two elements).

\medskip
Section~\ref{sec:mr.preliminaries} provides the basic setup
and Section~\ref{sec:mr.ratop} summarizes several (different)
constructions of linear representations for the rational
operations (sum, product and inverse).
In a first reading only the first two propositions are
important.
In Section~\ref{sec:mr.standardform} we develop the notation
to be able to formulate a \emph{standard form}
in Definition~\ref{def:mr.stdals}.
The main result is then Theorem~\ref{thr:mr.minchar}
(or Algorithm~\ref{alg:mr.minals2})
for the minimization in Section~\ref{sec:mr.minimizing}.
And finally, in Section~\ref{sec:mr.applications},
some applications are mentioned.
Example~\ref{ex:mr.hua} can also serve as an
introduction for the work (by hand)
with linear representations.

\medskip
The intention of this paper is to be independent
of the other three papers in this series (about
the free field) as far as possible and leave it
to the reader, for example,
to interpret a standard form of the
inverse of a polynomial as its factorization
(into irreducible elements).
Although the idea for minimizing ``polynomial'' linear representations
is similar to that in Section~\ref{sec:mr.minimizing},
\cite[Algorithm~32]{Schrempf2017b9}
\ is only a very special case of Algorithm~\ref{alg:mr.minals2}
and the ``refinement'' in the former case is trivial.

Beside the algebraic approach presented here
there are analytical methods for solving the \emph{word problem}
(or testing \emph{rational identities} like in Example~\ref{ex:mr.hua})
in polynomial time by plugging in ``sufficiently large'' matrices
\cite{Garg2016a,Ivanyos2018a}
. Closely related to linear representations
are \emph{realizations} which can be
``cut down'' by plugging in operators
\cite{Helton2018a}
\ or ``reduced'' by plugging in matrices
\cite{Volcic2018a}
.
Yet another point of view is from \emph{invariant subspaces}
\cite{Gohberg2006a,Helton2018b}
.

Once the rich structure of Cohn and Reutenauer's
\emph{normal form} becomes ``visible'', a lot can
be done, transforming the rather abstract free field
into an \emph{applied}.
One can use non-commutative ``rational functions'' just like
rational numbers respectively ``free fractions'' just like
``classical'' fractions \ldots

\section{Preliminaries}\label{sec:mr.preliminaries}

We represent elements (in free fields) by
\emph{admissible linear systems} (Definition~\ref{def:mr.als}),
which are just a special form of \emph{linear representations}
(Definition~\ref{def:mr.rep}). Rational operations
(scalar multiplication, addition, multiplication, inverse) can be
easily formulated in terms of linear representations
(Proposition~\ref{pro:mr.ratop}).

\medskip

\begin{notation}
The set of the natural numbers is denoted by $\numN = \{ 1,2,\ldots \}$.
Zero entries in matrices are usually replaced by (lower) dots
to emphasize the structure of the non-zero entries
unless they result from transformations where there
were possibly non-zero entries before.
We denote by $I_n$ the identity matrix
and $\perm_n$ the permutation matrix that reverses the order
of rows/columns (of size $n$)
respectively $I$ and $\perm$ if the size is clear from the context.
\end{notation}

\medskip
Let $\field{K}$ be a \emph{commutative} field,
$\aclo{\field{K}}$ its algebraic closure and
$X = \{ x_1, x_2, \ldots, x_d\}$ be a \emph{finite} (non-empty) alphabet.
$\freeALG{\field{K}}{X}$ denotes the \emph{free associative
algebra} (or \emph{free $\field{K}$-algebra})
and $\field{F} = \freeFLD{\field{K}}{X}$ its \emph{universal field of
fractions} (or ``free field'') 
\cite{Cohn1995a,Cohn1999a}
. An element in $\freeALG{\field{K}}{X}$ is called (non-commutative or nc)
\emph{polynomial}.
In our examples the alphabet is usually $X=\{x,y,z\}$.
Including the algebra of \emph{nc rational series}
we have the following chain of inclusions:
\begin{displaymath}
\field{K}\subsetneq \freeALG{\field{K}}{X}
  \subsetneq \field{K}^{\text{rat}}\langle\!\langle X\rangle\!\rangle
  \subsetneq \freeFLD{\field{K}}{X} =: \field{F}.
\end{displaymath}
The \emph{free monoid} $X^*$ generated by $X$
is the set of all \emph{finite words}
$x_{i_1} x_{i_2} \cdots x_{i_n}$ with $i_k \in \{ 1,2,\ldots, d \}$.
An element of the alphabet is called \emph{letter},
one of the free monoid \emph{word}.
The multiplication on $X^*$ is the \emph{concatenation}
of words, that is,
$(x_{i_1} \cdots x_{i_m})\cdot (x_{j_1} \cdots x_{j_n})
= x_{i_1} \cdots x_{i_m} x_{j_1} \cdots x_{j_n}$,
with neutral element $1$, the \emph{empty word}.
For detailed introductions see
\cite[Chapter~1]{Berstel2011a}
\ or
\cite[Section~I.1]{Salomaa1978a}
.

\begin{definition}[Inner Rank, Full Matrix, Hollow Matrix
\cite{Cohn1985a,Cohn1999a}
]\label{def:mr.full}
Given a matrix $A \in \freeALG{\field{K}}{X}^{n \times n}$, the \emph{inner rank}
of $A$ is the smallest number $m\in \numN$
such that there exists a factorization
$A = T U$ with $T \in \freeALG{\field{K}}{X}^{n \times m}$ and
$U \in \freeALG{\field{K}}{X}^{m \times n}$.
The matrix $A$ is called \emph{full} if $m = n$,
\emph{non-full} otherwise.
It is called \emph{hollow} if it contains a zero submatrix of size
$k \times l$ with $k+l>n$.
\end{definition}

\begin{definition}[Associated and Stably Associated Matrices
\cite{Cohn1995a}
]\label{def:mr.ass}
Two matrices $A$ and $B$ over $\freeALG{\field{K}}{X}$ (of the same size)
are called \emph{associated} over a subring $R\subseteq \freeALG{\field{K}}{X}$ 
if there exist
invertible matrices $P,Q$ over $R$ such that
$A = P B Q$. $A$ and $B$ (not necessarily of the same size)
are called \emph{stably associated}
if $A\oplus I_p$ and $B\oplus I_q$ are associated for some unit
matrices $I_p$ and $I_q$.
Here by $C \oplus D$ we denote the diagonal sum
$\bigl[\begin{smallmatrix} C & . \\ . & D \end{smallmatrix}\bigr]$.
\end{definition}

\begin{lemma}[%
\protect{\cite[Corollary~6.3.6]{Cohn1995a}
}]\label{lem:mr.cohn95.636}
A linear square matrix over $\freeALG{\field{K}}{X}$
which is not full is associated over $\field{K}$ to a linear
hollow matrix.
\end{lemma}

\begin{definition}[Linear Representations, Dimension, Rank
\cite{Cohn1994a,Cohn1999a}
]\label{def:mr.rep}
Let $f \in \field{F}$.
A \emph{linear representation} of $f$ is a triple $\pi_f = (u,A,v)$ with
$u \in \field{K}^{1 \times n}$, full
$A = A_0 \otimes 1 + A_1 \otimes x_1 + \ldots
+ A_d \otimes x_d$, that is, $A$ is invertible over $\field{F}$,
$\forall\,\ell\, A_\ell \in \field{K}^{n\times n}$,
$v \in \field{K}^{n\times 1}$ 
and $f = u A\inv v$.
The \emph{dimension} of $\pi_f$ is $\dim \, (u,A,v) = n$.
It is called \emph{minimal} if $A$ has the smallest possible dimension
among all linear representations of $f$.
The ``empty'' representation $\pi = (,,)$ is
the minimal one of $0 \in \field{F}$ with $\dim \pi = 0$.
Let $f \in \field{F}$ and $\pi$ be a \emph{minimal}
linear representation of $f$.
Then the \emph{rank} of $f$ is
defined as $\rank f = \dim \pi$.
\end{definition}

\begin{remark}
Cohn and Reutenauer define linear representations
slightly more general, namely $f = c + u A\inv v$ with
possibly non-zero $c \in \field{K}$ and call it
\emph{pure} when $c=0$. Two linear representations are called
\emph{equivalent} if they represent the same element
\cite{Cohn1999a}
.
Two (pure) linear representations $(u,A,v)$ and $(\tilde{u},\tilde{A},\tilde{v})$
of dimension $n$ are called \emph{isomorphic} if there exist invertible matrices
$P,Q \in \field{K}^{n \times n}$
such that $u = \tilde{u}Q$, $A = P \tilde{A} Q$ and $v=P\tilde{v}$
\cite{Cohn1999a}
.
\end{remark}

\begin{theorem}[%
\protect{\cite[Theorem 1.4]{Cohn1999a}
}]\label{thr:mr.cohn99.14}
If $\pi' = (u',A',v')$ and $\pi''=(u'',A'',v'')$ are equivalent
(pure) linear representations, of which the first is minimal,
then the second is isomorphic to a representation $\pi = (u,A,v)$
which has the block decomposition
\begin{displaymath}
u =
\begin{bmatrix}
. & u' & * 
\end{bmatrix},\quad
A =
\begin{bmatrix}
* & * & * \\
. & A' & * \\
. & . & * 
\end{bmatrix}
\quad\text{and}\quad
v = 
\begin{bmatrix}
* \\ v' \\ .
\end{bmatrix}.
\end{displaymath}
\end{theorem}

\begin{remark}
In principle, given $\pi''$ of dimension $n$,
one can look for invertible matrices $P,Q$
such that $P\pi''Q = (u''Q, PA''Q, Pv'')$
has the form of $\pi$ to minimize $\pi''$.
However, for an alphabet with $d$ letters and
a lower left block of zeros of size $k \times (n-k)$ we
would get $(d+1)k(n-k) +k$ \emph{polynomial} equations
with at most quadratic terms and two equations of
degree $n$ (to ensure invertibility of the transformation matrices)
in $2n^2$ \emph{commuting} unknowns. This is already
rather challenging for $n=5$. The goal is therefore
to use linear techniques as far as possible.
\end{remark}

\begin{definition}[Left and Right Families
\cite{Cohn1994a}
]\label{def:mr.family}
Let $\pi=(u,A,v)$ be a linear representation of $f \in \field{F}$
of dimension $n$.
The families $( s_1, s_2, \ldots, s_n )\subseteq \field{F}$
with $s_i = (A\inv v)_i$
and $( t_1, t_2, \ldots, t_n )\subseteq \field{F}$
with $t_j = (u A\inv)_j$
are called \emph{left family} and \emph{right family} respectively.
$L(\pi) = \linsp \{ s_1, s_2, \ldots, s_n \}$ and
$R(\pi) = \linsp \{ t_1, t_2, \ldots, t_n \}$
denote their linear spans (over $\field{K}$).
\end{definition}

\begin{proposition}[%
\protect{\cite[Proposition~4.7]{Cohn1994a}
}]\label{pro:mr.cohn94.47}
A representation $\pi=(u,A,v)$ of an element $f \in \field{F}$
is minimal if and only if both, the left family
and the right family are $\field{K}$-linearly independent.
In this case, $L(\pi)$ and $R(\pi)$ depend only on $f$.
\end{proposition}

\begin{definition}[Element Types
\cite{Schrempf2017c9}
]\label{def:mr.typele}
An element $f \in \field{F}$ is called \emph{of type} $(1,*)$
(respectively $(0,*)$) if $1 \in R(f)$, that is, $1 \in R(\pi)$
for some \emph{minimal} linear representation $\pi$ of $f$,
(respectively $1 \notin R(f)$).
It is called \emph{of type} $(*,1)$ (respectively $(*,0)$)
if $1 \in L(f)$ (respectively $1 \notin L(f)$).
Both subtypes can be combined.
\end{definition}

\begin{remark}
The following definition is a special case
of Cohn's more general \emph{admissible systems}
\cite[Section~7]{Cohn2006a}
\ and the slightly more general \emph{linear representations}
\cite{Cohn1994a}
.
\end{remark}

\begin{definition}[Admissible Linear Systems
\cite{Schrempf2017a9}
]\label{def:mr.als}
A linear representation $\als{A} = (u,A,v)$ of $f \in \field{F}$
is called \emph{admissible linear system} (ALS) for $f$
if $u=e_1=[1,0,\ldots,0]$.
The element $f$ is then the first component
of the (unique) solution vector $s = A\inv v$.
An ALS is also written as $A s = v$,
or if $v = [0,\ldots,0,\lambda]$ as $(1,A,\lambda)$.
Given a linear representation $\als{A} = (u,A,v)$
of dimension $n$ of $f \in \field{F}$
and invertible matrices $P,Q \in \field{K}^{n\times n}$,
the transformed $P\als{A}Q = (uQ, PAQ, Pv)$ is
again a linear representation (of $f$).
If $\als{A}$ is an ALS,
the transformation $(P,Q)$ is called
\emph{admissible} if the first row of $Q$ is $e_1 = [1,0,\ldots,0]$.
\end{definition}

\begin{remark}
The left family $(A\inv v)_i$ (respectively the right family $(u A\inv)_j$)
and the solution vector $s$ of $As = v$ (respectively $t$ of $u = tA$)
are used synonymously.
\end{remark}

Transformations can be done by elementary row- and column operations,
explained in detail in
\cite[Remark~1.12]{Schrempf2017a9}
. For further remarks and connections to the related concepts
of linearization and realization see
\cite[Section~1]{Schrempf2017a9}
. 

For elements in the free associative algebra $\freeALG{\field{K}}{X}$
a special form (with an upper unitriangular system matrix) can be used.
It plays a crucial role in the factorization of polynomials
because it allows to formulate a minimal polynomial multiplication
(Proposition~\ref{pro:mr.minmul}) and \emph{upper unitriangular}
transformation matrices (invertible by definition) suffice to find
all possible factors (up to trivial units).
For details we refer to 
\cite[Section~2]{Schrempf2017b9}
.

\begin{Remark}
While it was intended in 
\cite{Schrempf2017b9}
\ to derive a ``better'' standard form including knowledge about
the factorization it turned out later that this is not that easy in the
general case because of the necessity to distinguish several cases in
the multiplication 
\cite[Section~5]{Schrempf2017c9}
. The term ``pre-standard ALS'' (for polynomials) is now replaced
by \emph{polynomial ALS} (which is just a special form of a \emph{refined} ALS,
Section~\ref{sec:mr.standardform}). And a \emph{minimal} polynomial ALS
is already in \emph{standard form}.
Especially to avoid confusion with special transformation matrices
for the factorization everything is put into a uniform context
in \cite{Schrempf2018b}
. For an overview see also 
\cite[Figure~1]{Schrempf2018c2}
.
There are only some minor changes in the formulation of results and proofs
necessary, for example to construct the product in
\cite[Lemma~39]{Schrempf2017b9}
\ using $\frac{\lambda_2}{\lambda} \als{A}_1$ and
$\frac{\lambda}{\lambda_2}\als{A}_2$.
\end{Remark}

\begin{definition}[Polynomial ALS and Transformation
\cite{Schrempf2017b9}
]\label{def:mr.psals}
An ALS $\als{A} = (u,A,v)$ of dimension $n$
with system matrix $A = (a_{ij})$
for a non-zero polynomial $0 \neq p \in \freeALG{\field{K}}{X}$ 
is called
\emph{polynomial}, if
\begin{itemize}
\item[(1)] $v = [0,\ldots,0,\lambda]\trp$ for some $\lambda \in\field{K}$ and
\item[(2)] $a_{ii}=1$ for $i=1,2,\ldots, n$ and $a_{ij}=0$ for $i>j$,
  that is, $A$ is upper triangular.
\end{itemize}
An admissible transformation $(P,Q)$ for an ALS $\als{A}$
is called \emph{polynomial} if it has the form
\begin{displaymath}
(P,Q) = \left(
\begin{bmatrix}
1 & \alpha_{1,2} & \ldots & \alpha_{1,n-1} & \alpha_{1,n} \\
  & \ddots & \ddots & \vdots & \vdots \\
  &   & 1 & \alpha_{n-2,n-1} & \alpha_{n-2,n} \\
  &   &   & 1 & \alpha_{n-1,n} \\
  &   &   &   & 1
\end{bmatrix},
\begin{bmatrix}
1 & 0 & 0 & \ldots & 0 \\
  & 1 & \beta_{2,3} & \ldots & \beta_{2,n} \\
  &   & 1 & \ddots & \vdots  \\
  &   &   & \ddots & \beta_{n-1,n} \\
  &   &   &   & 1 \\
\end{bmatrix}
\right).
\end{displaymath}
If additionally $\alpha_{1,n} = \alpha_{2,n} = \ldots = \alpha_{n-1,n} = 0$
then $(P,Q)$ is called \emph{polynomial factorization transformation}.
\end{definition}

\begin{definition}\label{def:mr.reg}
Let $M = M_1 \otimes x_1 + \ldots + M_d \otimes x_d$
with $M_i \in \field{K}^{n \times n}$ for some $n\in \numN$.
An element in $\field{F}$ is called \emph{regular}
if it has a linear representation $(u,A,v)$ with $A = I - M$,
that is, $A_0 = I$ in Definition~\ref{def:mr.rep},
or equivalently, if $A_0$ is regular (invertible).
\end{definition}

\section{Rational Operations}\label{sec:mr.ratop}

Usually we want to construct \emph{minimal} admissible linear
systems (out of minimal ones), that is,
perform ``minimal'' rational operations.
Minimal scalar multiplication is trivial.
In some special cases \emph{minimal multiplication}
or even \emph{minimal addition} (if two elements are disjoint
\cite{Cohn1999a}
)
can be formulated (Proposition~\ref{pro:mr.minmul} or
\cite[Theorem~5.2]{Schrempf2017c9}
\ and 
\cite[Proposition~3.5]{Schrempf2017c9}
). For the \emph{minimal inverse} we have to distinguish four
cases, which are summarized in Theorem~\ref{thr:mr.mininv}.
In general we have to \emph{minimize} admissible linear systems.
This will be the main goal of the following sections.

\begin{proposition}[Minimal Monomial
\protect{\cite[Proposition~4.1]{Schrempf2017a9}
}]\label{pro:mr.minmon}
Let $k \in \numN$ and $f= x_{i_1} x_{i_2} \cdots x_{i_k}$ be a monomial in
$\freeALG{\field{K}}{X} \subseteq \freeFLD{\field{K}}{X}$.
Then
\begin{displaymath}
\als{A} = \left(
\begin{bmatrix}
1 & .  & \cdots & .
\end{bmatrix},
\begin{bmatrix}
1 & -x_{i_1} \\
  & 1 & -x_{i_2} \\
  & & \ddots & \ddots \\
  & & & 1 & -x_{i_k} \\
  & & & & 1
\end{bmatrix},
\begin{bmatrix}
. \\ .  \\ \vdots \\ .  \\ 1
\end{bmatrix}
\right)
\end{displaymath}
is a \emph{minimal} polynomial ALS of dimension $\dim(\als{A}) = k+1$.
\end{proposition}

\begin{proposition}[Rational Operations
\cite{Cohn1999a}
]\label{pro:mr.ratop}
Let $0\neq f,g \in \field{F}$ be given by the
admissible linear systems $\als{A}_f = (u_f, A_f, v_f)$
and $\als{A}_g = (u_g, A_g, v_g)$ respectively
and let $0\neq \mu \in \field{K}$.
Then admissible linear systems for the rational operations
can be obtained as follows:

\smallskip\noindent
The scalar multiplication
$\mu f$ is given by
\begin{displaymath}
\mu \als{A}_f =
\bigl( u_f, A_f, \mu v_f \bigr).
\end{displaymath}
The sum $f + g$ is given by
\begin{displaymath}
\als{A}_f + \als{A}_g =
\left(
\begin{bmatrix}
u_f & . 
\end{bmatrix},
\begin{bmatrix}
A_f & -A_f u_f\trp u_g \\
. & A_g
\end{bmatrix}, 
\begin{bmatrix} v_f \\ v_g \end{bmatrix}
\right).
\end{displaymath}
The product $fg$ is given by
\begin{displaymath}
\als{A}_f \cdot \als{A}_g =
\left(
\begin{bmatrix}
u_f & . 
\end{bmatrix},
\begin{bmatrix}
A_f & -v_f u_g \\
. & A_g
\end{bmatrix},
\begin{bmatrix}
. \\ v_g
\end{bmatrix}
\right).
\end{displaymath}
And the inverse $f\inv$ is given by
\begin{displaymath}
\als{A}_f\inv =
\left(
\begin{bmatrix}
1 & . 
\end{bmatrix},
\begin{bmatrix}
-v_f & A_f \\
. & u_f
\end{bmatrix},
\begin{bmatrix}
. \\ 1
\end{bmatrix}
\right).
\end{displaymath}
\end{proposition}

Since we need alternative constructions
(to that in Proposition~\ref{pro:mr.ratop})
for the product we state them here
in Propositions~\ref{pro:mr.mul2} and~\ref{pro:mr.mul1}.
Before, we need some technical results from
\cite{Schrempf2017a9}
\ and
\cite{Schrempf2017b9}
. However these are rearranged such that similarities become
more obvious and the flexibility in applications is increased.
In particular one can proof Lemma~\ref{lem:mr.forL} by applying
Lemma~\ref{lem:mr.rt1},
see
\cite[Lemma~3.8]{Schrempf2017c9}
.

\begin{lemma}[%
\protect{\cite[Lemma~25]{Schrempf2017b9}
}]\label{lem:mr.rt1}
Let $\als{A} = (u,A,v)$ be an ALS
of dimension $n \ge 1$ with $\field{K}$-linearly independent
left family $s = A\inv v$ and
$B = B_0 \otimes 1 + B_1 \otimes x_1 + \ldots + B_d \otimes x_d$
with $B_\ell\in \field{K}^{m\times n}$, such that
$B s = 0$. Then there exists a (unique) $T \in \field{K}^{m \times n}$
such that $B = TA$.
\end{lemma}

\begin{lemma}[for Type~$(*,1)$
\protect{\cite[Lemma~4.11]{Schrempf2017a9}
}]\label{lem:mr.forL}
Let $\als{A} = (u,A,v)$ be a \emph{minimal} ALS
with $\dim \als{A} =n \ge 2$ and $1\in L(\als{A})$. Then there
exists an admissible transformation $(P,Q)$ such that
the last row of $PAQ$ is $[0,\ldots,0,1]$
and $Pv = [0,\ldots,0,\lambda]\trp$
for some $\lambda \in \field{K}$.
\end{lemma}

\begin{lemma}[for Type~$(1,*)$
\protect{\cite[Lemma~4.12]{Schrempf2017a9}
}]\label{lem:mr.forR}
Let $\als{A} = (u,A,v)$ be a \emph{minimal} ALS
with $\dim \als{A} = n \ge 2$ and $1\in R(\als{A})$. Then there
exists an admissible transformation $(P,Q)$ such that
the first column of $PAQ$ is $[1,0,\ldots,0]\trp$
and $Pv = [0,\ldots,0,\lambda]\trp$
for some $\lambda \in \field{K}$.
\end{lemma}

\begin{remark}
If $g$ is of type $(*,1)$ then, by Lemma~\ref{lem:mr.forL},
\emph{each} minimal ALS for $g$ can be transformed into one
with the last row of the form $[0,\ldots,0,1]$.
If $g$ is of type $(1,*)$ then, by Lemma~\ref{lem:mr.forR},
\emph{each} minimal ALS for $g$ can be transformed into one
with the first column of the form $[1,0,\ldots,0]\trp$.
This can be done by \emph{linear} techniques,
see the remark before
\cite[Theorem~4.13]{Schrempf2017a9}
.
\end{remark}

\begin{remark}
Since $p \in \freeALG{\field{K}}{X}$ is of type $(1,1)$,
both constructions
can be used for the minimal polynomial multiplication
(Proposition~\ref{pro:mr.minmul}).
The alternative proof in 
\cite{Schrempf2017c9}
\ relies in particular on Lemma~\ref{lem:mr.slinin}.
One could call the multiplication from Proposition~\ref{pro:mr.ratop}
type $(*,*)$. For a discussion of \emph{minimality} of
different types of multiplication we refer to
\cite{Schrempf2017c9}
.
\end{remark}

\begin{proposition}[Multiplication Type $(1,*)$
\protect{\cite[Proposition~3.11]{Schrempf2017c9}
}]\label{pro:mr.mul2}
Let $f,g\in \field{F} \setminus \field{K}$ be given by the
admissible linear systems
$\als{A}_f = (u_f, A_f, v_f) = (1,A_f,\lambda_f)$ of dimension $n_f$
of the form
\begin{displaymath}
\als{A}_f = \left(
\begin{bmatrix}
1 & . & . 
\end{bmatrix},
\begin{bmatrix}
a & b' & b \\
a' & B & b'' \\
. & . & 1
\end{bmatrix},
\begin{bmatrix}
. \\ . \\ \lambda_f
\end{bmatrix}
\right)
\end{displaymath}
and $\als{A}_g = (u_g, A_g, v_g) = (1,A_g,\lambda_g)$
of dimension $n_g$ respectively.
Then an ALS for $fg$ of dimension $n = n_f + n_g - 1$ is given by
\begin{displaymath}
\als{A} = \left(
\begin{bmatrix}
1 & . & . 
\end{bmatrix},
\begin{bmatrix}
a & b' & \lambda_f b u_g \\
a' & B & \lambda_f b'' u_g \\
. & . & A_g
\end{bmatrix},
\begin{bmatrix}
. \\ . \\ v_g
\end{bmatrix}
\right).
\end{displaymath}
\end{proposition}

\begin{proposition}[Multiplication Type $(*,1)$
\protect{\cite[Proposition~3.12]{Schrempf2017c9}
}]\label{pro:mr.mul1}
Let $f,g\in \field{F} \setminus \field{K}$ be given by the
admissible linear systems $\als{A}_f = (u_f, A_f, v_f) = (1, A_f, \lambda_f)$ 
of dimension $n_f$ and
$\als{A}_g = (u_g, A_g, v_g) = (1, A_g, \lambda_g)$
of dimension $n_g$ of the form
\begin{displaymath}
\als{A}_g = \left(
\begin{bmatrix}
1 & . & .
\end{bmatrix},
\begin{bmatrix}
1 & b'& b \\
. & B & b'' \\
. & c' & c
\end{bmatrix},
\begin{bmatrix}
. \\ . \\ \lambda_g
\end{bmatrix}
\right)
\end{displaymath}
respectively.
Then an ALS for $fg$ of dimension $n = n_f + n_g -1$ is given by
\begin{displaymath}
\als{A} = \left(
\begin{bmatrix}
u_f & . & . 
\end{bmatrix},
\begin{bmatrix}
A_f & e_{n_f} \lambda_f b' & e_{n_f} \lambda_f b \\
. & B & b'' \\
. & c' & c
\end{bmatrix},
\begin{bmatrix}
. \\ . \\ \lambda_g
\end{bmatrix}
\right).
\end{displaymath}
\end{proposition}

\medskip
\begin{remark}
Notice that the transformation in the following lemma is \emph{not}
necessarily admissible. However, except for $n=2$ (which can be
treated by permuting the last two elements in the left family),
it can be chosen such that it is \emph{admissible}.
The proof is similar to that of 
\cite[Lemma~2.4]{Schrempf2017b9}
.
\end{remark}

\begin{lemma}[\protect{%
\cite[Lemma~3.15]{Schrempf2017c9}
}]\label{lem:mr.min1}
Let $\als{A} = (u,A,v)$ be an ALS of dimension $n\ge 2$
with $v = [0,\ldots,0,\lambda]\trp$ and
$\field{K}$-linearly dependent left family $s=A\inv v$.
Let $m \in \{ 2, 3, \ldots, n \}$ be the minimal index
such that the left subfamily $\underline{s} = (A\inv v)_{i=m}^n$
is $\field{K}$-linearly independent.
Let $A = (a_{ij})$ and assume that $a_{ii}=1$ for $1 \le i \le m$
and $a_{ij}=0$ for $j < i \le m$ (upper triangular $m \times m$ block)
and $a_{ij}=0$ for $j \le m < i$ (lower left zero block of size $(n-m) \times m$).
Then there exists matrices $T,U \in \field{K}^{1 \times (n+1-m)}$
such that
\begin{displaymath}
U + (a_{m-1,j})_{j=m}^n - T(a_{ij})_{i,j=m}^n  =
\begin{bmatrix}
  0 & \ldots & 0
\end{bmatrix}
\quad\text{and}\quad
T(v_i)_{i=m}^n = 0.
\end{displaymath}
\end{lemma}

\begin{lemma}[\protect{%
\cite[Lemma~3.16]{Schrempf2017c9}
}]\label{lem:mr.slinin}
Let $p \in \freeALG{\field{K}}{X} \setminus \field{K}$ and
$g \in \field{F} \setminus \field{K}$
be given by the \emph{minimal} admissible linear systems
$A_p = (u_p, A_p, v_p)$ and $A_g = (u_g, A_g, v_g)$
of dimension $n_p$ and $n_g$ respectively
with $1 \in R(g)$.
Then the left family of the ALS $\als{A} = (u,A,v)$
for $pg$ of dimension $n = n_p + n_g -1$ from
Proposition~\ref{pro:mr.mul1} is $\field{K}$-linearly independent.
\end{lemma}

\begin{proposition}[Minimal Polynomial Multiplication
\protect{\cite[Proposition~26]{Schrempf2017b9}
}]\label{pro:mr.minmul}
Let $0\ne p,q \in \freeALG{\field{K}}{X}$ be given by the
\emph{minimal} polynomial admissible linear systems
$A_p = (1, A_p, \lambda_p)$ and
$A_q = (1, A_q, \lambda_q)$ of dimension $n_p,n_q \ge 2$ respectively.
Then the ALS $\als{A}$ from Proposition~\ref{pro:mr.mul2} for $pq$
is \emph{minimal} of dimension $n = n_p + n_q - 1$.
\end{proposition}

\begin{theorem}[Minimal Inverse
\protect{\cite[Theorem~4.13]{Schrempf2017a9}
}]\label{thr:mr.mininv}
Let $f \in \field{F} \setminus \field{K}$ be given by
the \emph{minimal} admissible linear system $\als{A} = (u, A, v)$
of dimension $n$.
Then a \emph{minimal} ALS for $f\inv$ is given
in the following way:

\smallskip\noindent
$f$ of type $(1,1)$ yields $f\inv$ of type $(0,0)$ with $\dim(\als{A}') = n-1$:
\begin{displaymath}
\als{A}'=\left(1, 
\begin{bmatrix}
-\lambda \perm b'' & -\perm B \perm \\
-\lambda b & - b'\perm
\end{bmatrix},
1 \right) \quad\text{for}\quad
\als{A} = \left(1, 
\begin{bmatrix}
1 & b' & b \\
. & B & b'' \\
. & . & 1
\end{bmatrix},
\lambda \right).
\end{displaymath}
$f$ of type $(1,0)$ yields $f\inv$ of type $(1,0)$ with $\dim(\als{A}') = n$:
\begin{displaymath}
\als{A}' = \left(1,
\begin{bmatrix}
1 & - \frac{1}{\lambda} c & - \frac{1}{\lambda} c' \perm \\
. & -\perm b'' & -\perm B \perm \\
. & -b & -b'\perm 
\end{bmatrix},
1 \right)
\quad\text{for}\quad
\als{A} = \left(1,
\begin{bmatrix}
1 & b' & b \\
. & B & b'' \\
. & c' & c
\end{bmatrix},
\lambda \right).
\end{displaymath}
$f$ of type $(0,1)$ yields $f\inv$ of type $(0,1)$ with $\dim(\als{A}') = n$:
\begin{displaymath}
\als{A}' = \left( 1,
\begin{bmatrix}
-\lambda \perm b'' & -\perm B \perm & -\perm a' \\
-\lambda b & -b'\perm & -a \\
. & . & 1 
\end{bmatrix},
1 \right)
\quad\text{for}\quad
\als{A} = \left(1,
\begin{bmatrix}
a & b' & b \\
a' & B & b'' \\
. & . & 1
\end{bmatrix},
\lambda \right).
\end{displaymath}
$f$ of type $(0,0)$ yields $f\inv$ of type $(1,1)$ with $\dim(\als{A}') = n+1$:
\begin{displaymath}
\als{A}' = \left(1,
\begin{bmatrix}
\perm v & -\perm A \perm \\
. & u \perm
\end{bmatrix},
1 \right).
\end{displaymath}
(Recall that the permutation matrix $\perm$ reverses the order of rows/columns.)
\end{theorem}

\begin{corollary}
Let $p \in \freeALG{\field{K}}{X}$ with $\rank p = n \ge 2$.
Then $\rank(p\inv) = n-1$.
\end{corollary}

\begin{corollary}\label{cor:mr.rank1}
Let $0 \neq f \in \field{F}$.
Then $f \in \field{K}$ if and only if $\rank(f) = \rank(f\inv) = 1$.
\end{corollary}

\begin{remark}
Notice that $n\ge 2$ for type $(1,1)$, $(1,0)$ and $(0,1)$.
The block $B$ is always square of size $n-2$.
For $n=2$ the system matrix of $\als{A}$ is
\begin{itemize}
\item
  $\bigl[\begin{smallmatrix} 1 & b \\ . & 1 \end{smallmatrix}\bigr]$
  for type $(1,1)$,
\item
  $\bigl[\begin{smallmatrix} 1 & b \\ . & c \end{smallmatrix}\bigr]$
  for type $(1,0)$ and
\item
  $\bigl[\begin{smallmatrix} a & b \\ . & 1 \end{smallmatrix}\bigr]$
  for type $(0,1)$.
\end{itemize}
\end{remark}

\section{A Standard Form}\label{sec:mr.standardform}

After providing a ``language'' to be able to formulate
``operations'' on the system matrix of an admissible linear
system we can define a \emph{standard form} (Definition~\ref{def:mr.stdals}).
A standard ALS will be \emph{minimal} and (has) \emph{refined} (pivot blocks).
In the following section we will minimize a refined ALS.
This is somewhat technical to describe but simple (we
need to solve \emph{linear} systems of equations).
At the end of this section we will illustrate how to
refine pivot blocks. To the contrary this is easy
to describe but in general very difficult to accomplish
(we need to solve \emph{polynomial} systems of equations).
Since the latter is closely related to factorization we
refer to 
\cite{Schrempf2017b9}
\ and
\cite{Schrempf2017c9}
\ for further information.

\begin{definition}[Pivot Blocks, Block Transformation]\label{def:mr.pivot}
Let $\als{A} = (u,A,v)$ be an admissible linear system
and denote $A = (A_{ij})_{i,j=1}^m$ the block decomposition
(with square diagonal blocks $A_{ii}$) with \emph{maximal} $m$
such that $A_{ij}=0$ for $i > j$.
The diagonal blocks $A_{ii}$ are called \emph{pivot blocks},
the number $m$ is denoted by $\pivot(\als{A}) = \pivot(A)$.
The \emph{dimension} (or \emph{size})
of a pivot block $A_{ii}$ for $i \in \{ 1, 2, \ldots, m \}$
is $\dim_i (\als{A})$.
For $i<1$ or $i>m$ let $\dim_i(\als{A}) = 0$.
A (admissible) transformation $(P,Q)$ is called
(admissible) \emph{block transformation} (for $\als{A}$)
if $P_{ij} = Q_{ij} = 0$ for $i > j$
(with respect to the block structure of $\als{A}$).
\end{definition}

\begin{notation}
Let $\als{A} = (u,A,v)$ be an ALS with $m$ pivot blocks of size
$n_i = \dim_i(\als{A})$.
We denote by 
$n_{i:j} = \dim_{i:j}(\als{A}) = n_i + n_{i+1} + \ldots + n_j$
the sum of the sizes of the pivot blocks $A_{ii}$ to $A_{jj}$
(with the convention $n_{i:j} = 0$ for $j < i$).
For a given system the identity matrix of size $n_{i:j}$
is denoted by $I_{i:j}$.
If $(P,Q)$ is an admissible transformation for $\als{A}$
then $(PAQ)_{ij}$ denotes the (to the block decomposition of $A$
corresponding) block $(i,j)$ of size $n_i \times n_j$ in $PAQ$,
$(Pv)_{\block{i}}$ that of size $n_i \times 1$ in $Pv$ and
$(uQ)_{\block{j}}$ that of size $1 \times n_j$ in $uQ$.
\end{notation}

\begin{notation}
Components in the left family $s = A\inv v$ are (as usual)
denoted by $s_i$. The $j$-th component for $1 \le j \le \dim_i(\als{A})$
of the $i$-th block for $1 \le i \le \pivot(\als{A})$ is
denoted by $s_{\block{i}(j)}$.
A subfamily of $s$ with respect to the pivot block $k$
is denoted by $s_{\block{k}}$,
$s_{\block{i:j}} = (s_{\block{i}}, s_{\block{i+1}}, \ldots, s_{\block{j}})$.
Analogous is used for the right family $t =u A\inv$.
\end{notation}

\begin{notation}
A ``grouping'' of pivot blocks $\{ i, i+1, \ldots, j \}$ 
of the system matrix is denoted by $A_{i:j,i:j}$.
If it is clear from the context where the block ends (respectively starts)
we write $A_{i:,i:}$ (respectively $A_{:j,:j}$),
in particular with respect to a given pivot block.
For example $A_{1:,1:}$, $A_{k,k}$ and $A_{:m,:m}$.
\end{notation}

\begin{definition}[Admissible Pivot Block Transformation]\label{def:mr.pivottrn}
Let $\als{A} = (u,A,v)$ be an ALS with $m = \pivot(\als{A})$
pivot blocks of size $n_i = \dim_i(\als{A})$.
An admissible transformation $(P,Q)$ of the form
$( I_{1:k-1}\oplus \bar{T} \oplus I_{k+1:m}, I_{1:k-1}\oplus \bar{U} \oplus I_{k+1:m})$
with $\bar{T},\bar{U} \in \field{K}^{n_k \times n_k}$ is called
\emph{admissible for pivot block $k$} or (admissible)
\emph{$k$-th pivot block transformation}.
\end{definition}

\begin{definition}[Refined Pivot Block and Refined ALS]\label{def:mr.redpivot}
Let $\als{A} = (u,A,v)$ be an ALS with $m = \pivot(\als{A})$
pivot blocks of size $n_i = \dim_i(\als{A})$.
A pivot block $A_{kk}$ (for $1 \le k \le m$)
is called \emph{refined}
if there does not exist an admissible pivot block transformation
$(P,Q)_k$ such that $(PAQ)_{kk}$ has a lower left
block of zeros of size $i \times (n_k -i)$
for an $i \in \{1, 2, \ldots, n_k-1\}$.
The admissible linear system $\als{A}$ is called \emph{refined}
if all pivot blocks are refined.
\end{definition}

\begin{remarks}
That the ``form'' of a refined ALS is not unique,
it will be illustrated in the following example.
Moreover, a refined ALS is not necessarily refined over
$\aclo{\field{K}}$: Let
\begin{displaymath}
\als{A} = \left(
\begin{bmatrix}
1 & . 
\end{bmatrix},
\begin{bmatrix}
1 & x \\
2x & 1 
\end{bmatrix},
\begin{bmatrix}
. \\ 1
\end{bmatrix}
\right).
\end{displaymath}
Adding $\sqrt{2}$-times row~1 to row~2 and
subtracting $\sqrt{2}$-times column~2 of column~1
yields
\begin{displaymath}
\begin{bmatrix}
1 -\sqrt{2} x & x \\
0 & 1 + \sqrt{2}x
\end{bmatrix}
s =
\begin{bmatrix}
. \\ 1
\end{bmatrix}.
\end{displaymath}
See also 
\cite[Example~37]{Schrempf2017b9}
.
\end{remarks}

\begin{Example}\label{ex:mr.verfeinert}
Let $f = (y\inv - x)\inv$ be given by the minimal ALS
\begin{displaymath}
\als{A}_f = \left(
\begin{bmatrix}
1 & .
\end{bmatrix},
\begin{bmatrix}
1 & -y \\
-x & 1
\end{bmatrix},
\begin{bmatrix}
. \\ 1
\end{bmatrix}
\right).
\end{displaymath}
For $f +3z$ we have the (minimal) ALS
\begin{displaymath}
\als{A} = \left(
\begin{bmatrix}
1 & . & . & . 
\end{bmatrix},
\begin{bmatrix}
1 & -y & -1 & . \\
-x & 1 & x & . \\
. & . & 1 & -z \\
. & . & . & 1 
\end{bmatrix},
\begin{bmatrix}
. \\ 1 \\ . \\ 3
\end{bmatrix}
\right).
\end{displaymath}
Since $1 \in R(\als{A})$ and $\als{A}$ is constructed by
the addition in Proposition~\ref{pro:mr.ratop}, it can
easily be transformed ---in a controlled way--- into
another ALS with \emph{refined} pivot block structure.
Firstly we add column~3 to column~1,
\begin{displaymath}
\als{A}' = \left(
\begin{bmatrix}
1 & . & . & .
\end{bmatrix},
\begin{bmatrix}
0 & -y & -1 & . \\
0 & 1 & x & . \\
1 & . & 1 & -z \\
. & . & . & 1 
\end{bmatrix},
\begin{bmatrix}
. \\ 1 \\ . \\ 3
\end{bmatrix}
\right),
\end{displaymath}
and then we exchange rows~1 and~3:
\begin{displaymath}
\als{A}'' = \left(
\begin{bmatrix}
1 & . & . & .
\end{bmatrix},
\begin{bmatrix}
1 & . & 1 & -z \\
. & 1 & x & . \\
. & -y & -1 & . \\
. & . & . & 1 
\end{bmatrix},
\begin{bmatrix}
. \\ 1 \\ . \\ 3
\end{bmatrix}
\right).
\end{displaymath}
\end{Example}

\begin{definition}[Block Decomposition of an ALS, Block Row and Column Transformation]
\label{def:mr.bloals}
Let $\als{A} = (u,A,v)$ be an ALS of dimension~$n$
with $m = \pivot(\als{A}) \ge 2$ pivot blocks of size $n_i = \dim_i(\als{A})$.
For $1 \le k \le m$ the \emph{block decomposition}
with respect to the pivot block~$k$ is the system
\begin{displaymath}
\als{A}^{[\block{k}]} = \left(
\begin{bmatrix}
u_{\block{1:}} & . & . 
\end{bmatrix},
\begin{bmatrix}
A_{1:,1:} & A_{1:,k} & A_{1:,:m} \\
. & A_{k,k} & A_{k,:m} \\
. & . & A_{:m,:m}
\end{bmatrix},
\begin{bmatrix}
v_{\block{1:}} \\ v_{\block{k}} \\ v_{\block{:3}}
\end{bmatrix}
\right)
\end{displaymath}
with (square) diagonal blocks $A_{1:,1:}$, $A_{k,k}$ and $A_{:m,:m}$
of size $n_{1:k-1}$, $n_k$ respectively $n_{k+1:m}$.
($\block{k}$ is used here to emphasize that $k$ is a block index.)
By $\als{A}^{[-\block{k}]}$ we denote the ALS $\als{A}^{[\block{k}]}$
without block row/column~$k$ of dimension $n-n_k$
(not necessarily equivalent to $\als{A}^{[\block{k}]}$):
\begin{displaymath}
\als{A}^{[-\block{k}]} = \left(
\begin{bmatrix}
u_{\block{1}} & . 
\end{bmatrix},
\begin{bmatrix}
A_{1:,1:} & A_{1:,:m} \\
. & A_{:m,:m}
\end{bmatrix},
\begin{bmatrix}
v_{\block{1:}} \\ v_{\block{:3}}
\end{bmatrix}
\right).
\end{displaymath}
An admissible transformation
$(P,Q)_{\block{k}} = \bigl(P(\bar{T},T), Q(\bar{U},U)\bigr)\mthstrut_{\block{k}}$
of the form
\begin{equation}\label{eqn:mr.blorowtrn}
(P,Q)_{\block{k}} = \left(
\begin{bmatrix}
I_{1:k-1} & . & . \\
. & \bar{T} & T \\
. & . & I_{k+1:m}
\end{bmatrix},
\begin{bmatrix}
I_{1:k-1} & . & . \\
. & \bar{U} & U \\
. & . & I_{k+1:m}
\end{bmatrix}
\right)
\end{equation}
is called $k$-th \emph{block row transformation}
for $\als{A}^{[\block{k}]}$,
one of the form $\bigl(P(\bar{T},T),Q(\bar{U},U)\bigr)\mthstrut^{\block{k}}$,
\begin{equation}\label{eqn:mr.blocoltrn}
(P,Q)^{\block{k}} = \left(
\begin{bmatrix}
I_{1:k-1} & T & . \\
. & \bar{T} & . \\
. & . & I_{k+1:m}
\end{bmatrix},
\begin{bmatrix}
I_{1:k-1} & U & . \\
. & \bar{U} & . \\
. & . & I_{k+1:m}
\end{bmatrix}
\right)
\end{equation}
is called $k$-th \emph{block column transformation}
for $\als{A}^{[\block{k}]}$.
For $\bar{T}=\bar{U} = I_{n_k}$ we write also
$P(T)$ respectively $Q(U)$ and call the block transformation
\emph{particular}.
\end{definition}

\begin{definition}[Standard Admissible Linear System]\label{def:mr.stdals}
A \emph{minimal} and \emph{refined}
ALS $\als{A} = (u,A,v) = (1,A,\lambda)$,
that is, $v = [0,\ldots, 0, \lambda]$, is called
\emph{standard}.
\end{definition}

\begin{remark}
For a polynomial $p$ given by a \emph{standard} ALS
$\als{A}$ (of dimension $n \ge 2$) the minimal inverse
of $\als{A}$ (of dimension $n-1$) is refined if and
only if $\als{A}$ is obtained by the minimal polynomial
multiplication of its irreducible factors $q_i$ in
$p = q_1 q_2 \cdots q_m$.
For a detailed discussion about the factorization of polynomials
(in free associative algebras) we refer to
\cite{Schrempf2017b9}
.
One of the simplest non-trivial examples is
\begin{displaymath}
p = x(1-yx) = x-xyx = (1-xy)x.
\end{displaymath}
With respect of the general factorization theory 
it is open to show that the free field
is a ``similarity unique factorization domain''
\cite{Schrempf2017c9}
.
\end{remark}

\begin{Example}[Factorization versus Refinement]
Later, in Example~\ref{ex:mr.hua} (Hua's identity),
we will need to refine a pivot block. Although
the necessary transformations there are obvious,
the procedure should be illustrated in a systematic way.
But first we have a look on \emph{how} this $2\times 2$
block appears, namely by inverting the element given by the
ALS
\begin{displaymath}
\als{A} = \left(
\begin{bmatrix}
1 & . & . 
\end{bmatrix},
\begin{bmatrix}
x & 0 & 1 \\
. & 1 & y \\
. & x & 1
\end{bmatrix},
\begin{bmatrix}
0 \\ . \\ -1
\end{bmatrix}
\right).
\end{displaymath}
If we exchange columns~2 and~3 it is immediate that this is
the ``product'' of the admissible linear systems
\begin{displaymath}
\als{A}_1 = \left(
\begin{bmatrix}
1
\end{bmatrix},
\begin{bmatrix}
x
\end{bmatrix},
\begin{bmatrix}
-1
\end{bmatrix}
\right)\quad\text{and}\quad
\als{A}_2 = \left(
\begin{bmatrix}
1 & . 
\end{bmatrix},
\begin{bmatrix}
y & 1 \\
1 & x
\end{bmatrix},
\begin{bmatrix}
. \\ -1
\end{bmatrix}
\right).
\end{displaymath}
Applying the minimal inverse on
\begin{displaymath}
\als{A}' = \left(
\begin{bmatrix}
1 & . & . 
\end{bmatrix},
\begin{bmatrix}
x & 1 & 0 \\
. & y & 1 \\
. & 1 & x
\end{bmatrix},
\begin{bmatrix}
0 \\ . \\ -1
\end{bmatrix}
\right),
\end{displaymath}
we get a \emph{refined} (and minimal) ALS, namely
\begin{displaymath}
\als{A}'' = \left(
\begin{bmatrix}
1 & . & . & . 
\end{bmatrix},
\begin{bmatrix}
-1 & -x & -1 & . \\
. & -1 & -y & . \\
. & . & -1 & -x \\
. & . & . & 1
\end{bmatrix},
\begin{bmatrix}
. \\ . \\ . \\ 1
\end{bmatrix}
\right).
\end{displaymath}
The factorization here is really simple.
\end{Example}

\begin{Example}[Pivot Block Refinement]
Now we focus on the refinement of the second pivot block in
the ALS
\begin{displaymath}
\als{A} = \left(
\begin{bmatrix}
1 & . & . & .
\end{bmatrix},
\begin{bmatrix}
1 & 1 & x & . \\
. & y & 1 & . \\
. & 1 & 0 & x \\
. & . & . & 1
\end{bmatrix},
\begin{bmatrix}
. \\ . \\ . \\ 1
\end{bmatrix}
\right).
\end{displaymath}
(This is the ALS~\eqref{eqn:mr.hua2} with the first three rows
scaled by $-1$.)
We are looking for an admissible transformation $(P,Q)$ of the form
\begin{displaymath}
(P,Q) = \left(
\begin{bmatrix}
1 & . & . & . \\
. & \alpha_{2,2} & \alpha_{2,3} & . \\
. & \alpha_{3,2} & \alpha_{3,3} & . \\
. & . & . & 1
\end{bmatrix},
\begin{bmatrix}
1 & . & . & . \\
. & \beta_{2,2} & \beta_{2,3} & . \\
. & \beta_{3,2} & \beta_{3,3} & . \\
. & . & . & 1
\end{bmatrix}
\right).
\end{displaymath}
In particular these matrices $P$ and $Q$ have to be
\emph{invertible}, that is, we need the conditions
$\det(P) \neq 0$ and $\det(Q) \neq 0$.
To create a lower left $1 \times 1$ block of zeros in
$(PAQ)_{2,2}$ we need to solve the following \emph{polynomial}
system of equations (with \emph{commuting} unknowns $\alpha_{ij}$
and $\beta_{ij}$):
\begin{align*}
\alpha_{2,2} \alpha_{3,3} - \alpha_{2,3} \alpha_{3,2} &= 1, \\
\beta_{2,2} \beta_{3,3} - \beta_{2,3} \beta_{3,2} &= 1, \\
\alpha_{3,2} \beta_{3,2} + \alpha_{3,3}\beta_{2,2} &= 0
  \quad\text{for $1$, and}\\
\alpha_{3,2} \beta_{2,2} &= 0
  \quad\text{for $y$}.
\end{align*}
We obtain the last two equations by multiplication
of the transformation blocks with the corresponding
coefficient matrices of the pivot blocks
(irrelevant equations are marked with ``$*$'' on the right hand side)
\begin{align*}
\begin{bmatrix}
\alpha_{2,2} & \alpha_{2,3} \\
\alpha_{3,2} & \alpha_{3,3}
\end{bmatrix}
\begin{bmatrix}
. & 1 \\
1 & .
\end{bmatrix}
\begin{bmatrix}
\beta_{2,2} & \beta_{2,3} \\
\beta_{3,2} & \beta_{3,3}
\end{bmatrix}
&=
\begin{bmatrix}
* & * \\
0 & *
\end{bmatrix}
\quad\text{for $1$, and}\\
\begin{bmatrix}
\alpha_{2,2} & \alpha_{2,3} \\
\alpha_{3,2} & \alpha_{3,3}
\end{bmatrix}
\begin{bmatrix}
1 & . \\
. & .
\end{bmatrix}
\begin{bmatrix}
\beta_{2,2} & \beta_{2,3} \\
\beta_{3,2} & \beta_{3,3}
\end{bmatrix}
&=
\begin{bmatrix}
* & * \\
0 & *
\end{bmatrix}
\quad\text{for $y$}.
\end{align*}
To solve this system of polynomial equations,
Gr\"obner--Shirshov bases 
\cite{Bokut2000a}
\ can be used.
For detailed discussions we refer to
\cite{Sturmfels2002a}
\ or
\cite{Cox2015a}
. The basic idea comes from
\cite[Theorem~4.1]{Cohn1999a}
\ and is also used in 
\cite[Proposition~42]{Schrempf2017b9}
\ and 
\cite[Section~5]{Schrempf2017c9}
.
For further remarks on the refinement of pivot blocks
see also 
\cite[Section~4.4]{Schrempf2018c2}
.
\end{Example}

\section{Minimizing a refined ALS}\label{sec:mr.minimizing}

First of all we derive the \emph{left} respectively \emph{right}
block minimization equations.
For that we consider an admissible linear system
$\als{A} = (u,A,v)$ of dimension~$n$ with $m = \pivot(\als{A}) \ge 2$
pivot blocks of size $n_i = \dim_i(\als{A})$.
For $1 \le k < m$ we transform this system using the block row
transformation
$(P,Q) = \bigl(P(\bar{T},T), Q(\bar{U},U)\bigr)\mthstrut_{\block{k}}$,
namely
\begin{align*}
PAQ &= 
\begin{bmatrix}
I_{1:k-1} & . & . \\
. & \bar{T} & T \\
. & . & I_{k+1:m}
\end{bmatrix}
\begin{bmatrix}
A_{1:,1:} & A_{1:,k} & A_{1:,:m} \\
. & A_{k,k} & A_{k,:m} \\
. & . & A_{:m,:m}
\end{bmatrix}
\begin{bmatrix}
I_{1:k-1} & . & . \\
. & \bar{U} & U \\
. & . & I_{k+1:m}
\end{bmatrix}\\
&=
\begin{bmatrix}
A_{1:,1:} & A_{1:,k} & A_{1:,:m} \\
. & \bar{T} A_{k,k} & \bar{T} A_{k,:m} + T A_{:m,:m} \\
. & . & A_{:m,:m}
\end{bmatrix}
\begin{bmatrix}
I_{1:k-1} & . & . \\
. & \bar{U} & U \\
. & . & I_{k+1:m}
\end{bmatrix}\\
&=
\begin{bmatrix}
A_{1:,1:} & A_{1:,k} \bar{U} & A_{1:,k} U + A_{1:,:m} \\
. & \bar{T} A_{k,k} \bar{U} & \bar{T} A_{k,k} U + \bar{T} A_{k,:m} + T A_{:m,:m} \\
. & . & A_{:m,:m}
\end{bmatrix}\quad\text{and}\\
Pv &= 
\begin{bmatrix}
I_{1:k-1} & . & . \\
. & \bar{T} & T \\
. & . & I_{k+1:m}
\end{bmatrix}
\begin{bmatrix}
v_{\block{1:}} \\ v_{\block{k}} \\ v_{\block{:m}}
\end{bmatrix}
=
\begin{bmatrix}
v_{\block{1:}} \\ \bar{T} v_{\block{k}} + T v_{\block{:m}} \\ v_{\block{:m}}
\end{bmatrix}.
\end{align*}
Now we can read off a 
\emph{sufficient} condition for
$(Q\inv s)_{\block{k}} = 0^{n_k \times 1}$, namely
the existence of matrices $T,U \in \field{K}^{n_k \times n_{k+1:m}}$ and
invertible matrices $\bar{T},\bar{U} \in\field{K}^{n_k \times n_k}$
such that
\begin{displaymath}
\bar{T} A_{k,k} U + \bar{T} A_{k,:m} + T A_{:m,:m} = 0^{n_k \times n_{k+1:m}}
\quad\text{and}\quad
\bar{T} v_{\block{k}} + T v_{\block{:m}} = 0^{n_k \times 1}.
\end{displaymath}
Since $\bar{T}$ is invertible (as a diagonal block of an
invertible matrix $P$), this condition is equivalent to the
existence of matrices $T',U \in \field{K}^{n_k \times n_{k+1:m}}$
such that
\begin{equation}\label{eqn:mr.lmsys2}
A_{k,k} U + A_{k,:m} + \underbrace{\bar{T}\inv T}_{=: T'} A_{:m,:m}
  = 0^{n_k \times n_{k+1:m}}
\quad\text{and}\quad
v_{\block{k}} + \underbrace{\bar{T}\inv T}_{=: T'} v_{\block{:m}} = 0^{n_k \times 1}.
\end{equation}
Applying the block column transformation
$(P,Q) = \bigl(P(\bar{T},T), Q(\bar{U},U)\bigr)\mthstrut^{\block{k}}$,
we obtain
\begin{align*}
PAQ &= 
\begin{bmatrix}
I_{1:k-1} & T & . \\
. & \bar{T} & . \\
. & . & I_{k+1:m}
\end{bmatrix}
\begin{bmatrix}
A_{1:,1:} & A_{1:,k} & A_{1:,:m} \\
. & A_{k,k} & A_{k,:m} \\
. & . & A_{:m,:m}
\end{bmatrix}
\begin{bmatrix}
I_{1:k-1} & U & . \\
. & \bar{U} & . \\
. & . & I_{k+1:m}
\end{bmatrix}\\
&=
\begin{bmatrix}
A_{1:,1:} & A_{1:,k} +T A_{k,k}& A_{1:,:m} + T A_{k,:m} \\
. & \bar{T} A_{k,k} & \bar{T} A_{k,:m} \\
. & . & A_{:m,:m}
\end{bmatrix}
\begin{bmatrix}
I_{1:k-1} & U & . \\
. & \bar{U} & . \\
. & . & I_{k+1:m}
\end{bmatrix}\\
&=
\begin{bmatrix}
A_{1:,1:} & A_{1:,1:} U + A_{1:,k} \bar{U} + T A_{k,k} \bar{U} & A_{1:,:m} + T A_{k,:m} \\
. & \bar{T} A_{k,k} \bar{U} & \bar{T} A_{k,:m} \\
. & . & A_{:m,:m}
\end{bmatrix}
\end{align*}
and therefore 
a \emph{sufficient} condition for $(t P\inv)_{\block{k}} = 0^{1 \times n_k}$,
namely the existence of matrices $T,U' \in \field{K}^{n_{1:k-1} \times n_k}$
such that
\begin{equation}\label{eqn:mr.rmsys2}
A_{1:,1:} \underbrace{U \bar{U}\inv}_{=:U'} + A_{1:,k} + T A_{k,k}
  = 0^{n_{1:k-1} \times n_k}.
\end{equation}

\begin{remark}
A variant of the \emph{linear} system of equations~\eqref{eqn:mr.lmsys2}
also appears in 
\cite[Lemma~2.3]{Schrempf2017a9}
\ and
\cite[Theorem~2.4]{Schrempf2017a9}
\ (linear word problem).
\end{remark}

\begin{Remark}[Extended ALS]\label{rem:mr.extals}
In some cases it is necessary to use an \emph{extended} ALS
to be able to apply all necessary \emph{left} minimization steps,
for example, for $f\inv f$ if $f$ is of type~$(1,1)$.
Let $\als{A} = (u,A,v) = (1,A,\lambda)$ be an ALS with
$m = \pivot(\als{A}) \ge 2$ pivot blocks and $k=1$.
The ``extended'' block decomposition is then
(the block row $A_{1:,1:}$ vanishes)
\begin{displaymath}
\als{A}^{[\block{k}]} = \left(
\hspace{1.28em}\left|\hspace{-1.6em}%
\begin{bmatrix}
1 & . & .
\end{bmatrix},
\right.
\hspace{1.68em}\left|\hspace{-2em}%
\begin{bmatrix}
1 & A_{0,k} &  . \\\hline\tabstrut
. & A_{k,k} & A_{k,:m} \\
. & . & A_{:m,:m}
\end{bmatrix},
\begin{bmatrix}
. \\\hline\tabstrut . \\ v_{\block{:m}}
\end{bmatrix}
\right.
\right)
\end{displaymath}
with $A_{0,k} = [-1,0,\ldots, 0]$.
The first row in $\als{A}^{[\block{1}]}$ is only changed indirectly
(via admissible column operations) and stays scalar.
Therefore it can be removed easily (if necessary).
This is illustrated in Example~\ref{ex:mr.alg}.
\end{Remark}

\begin{notation}
Given an admissible linear system $\als{A}$,
by
$\tilde{\als{A}} = \als{A}^{[+0]}$ we denote the
(to $\als{A}$ equivalent) \emph{extended} ALS.
Conversely,
\raisebox{0pt}[0pt][0pt]{%
$\tilde{\als{A}}^{[-0]} = \bigl(\als{A}^{[+0]}\bigr)\mthstrut^{[-0]}=\als{A}$}.
The additional row and column is indexed by~0.
If $\tilde{\als{A}}$ is transformed admissibly,
$\tilde{\als{A}}^{[-0]}$ is an ALS.
\end{notation}

\begin{definition}[Minimization Equations and Transformations]
\label{def:mr.meqn}
Let $\als{A} = (u,A,v)$ be an ALS of dimension $n$ with
$m = \pivot(\als{A}) \ge 2$ pivot blocks of size $n_i = \dim_i(\als{A})$.
For $k \in \{ 1,2,\ldots, m-1 \}$ the equations~\eqref{eqn:mr.lmsys2},
\begin{displaymath}
A_{k,k} U + A_{k,:m} + T A_{:m,:m} = 0^{n_k \times n_{k+1:m}}
\quad\text{and}\quad
v_{\block{k}} + T v_{\block{:m}} = 0^{n_k \times 1}
\end{displaymath}
with respect to the block decomposition $\als{A}^{[\block{k}]}$
and the particular block row transformation
$\bigl(P(T), Q(U)\bigr)\mthstrut_{\block{k}}$
are called \emph{left block minimization equations}.
They are denoted by
$\mathcal{L}_{\block{k}} = \mathcal{L}_{\block{k}}(\als{A})$.
A solution by the block row pair $(T,U)$ is denoted by
$\mathcal{L}_{\block{k}}(T,U) = 0$.
For $k \in \{ 2,3,\ldots, m \}$ the equations~\eqref{eqn:mr.rmsys2},
\begin{displaymath}
A_{1:,1:} U + A_{1:,k} + T A_{k,k} = 0^{n_{1:k-1} \times n_k}
\end{displaymath}
with respect to the block decomposition $\als{A}^{[\block{k}]}$
and the particular block column transformation
$\bigl(P(T), Q(U)\bigr)\mthstrut^{\block{k}}$
are called \emph{right block minimization equations}.
They are denoted by
$\mathcal{R}_{\block{k}} = \mathcal{R}_{\block{k}}(\als{A})$.
A solution by the block column pair $(T,U)$ is denoted by
$\mathcal{R}_{\block{k}}(T,U) = 0$.
\end{definition}

In the following example we have a close look
on the role of the factorization and how to avoid
the use of possibly non-linear techniques.
All the steps are explained in detail and correspond
(with exception of the solution of the linear
systems of equations) to that in the following
algorithm.

\begin{Example}\label{ex:mr.alg}
For $f = x\inv (1-xy)\inv$ and $g = x$ we consider
$h = fg = (1-yx)\inv$ given by the (non-minimal) ALS
(constructed by Proposition~\ref{pro:mr.mul1})
\begin{displaymath}
\als{A} = (u,A,v) = \left(
\begin{bmatrix}
1 & . & . & .
\end{bmatrix},
\begin{bmatrix}
x & 1 & . & . \\
. & y & -1 & . \\
. & -1 & x & -x \\
. & . & . & 1
\end{bmatrix},
\begin{bmatrix}
. \\ . \\ . \\ 1
\end{bmatrix}
\right),
\end{displaymath}
whose pivot blocks are \emph{refined}.
Here there exists an admissible transformation
(with $T=0$, $U=1$ and \emph{invertible} blocks
$\bar{T},\bar{U} \in \field{K}^{3 \times 3}$)
\begin{displaymath}
(P,Q) = \left(
\begin{bmatrix}
1 & 0 & 0 & . \\
0 & 1 & 0 & . \\
1 & 0 & 1 & T \\
. & . & . & 1
\end{bmatrix},
\begin{bmatrix}
1 & 0 & 0 & . \\
0 & 1 & 0 & . \\
-1 & 0 & 1 & U \\
. & . & . & 1
\end{bmatrix}
\right),
\end{displaymath}
that yields the ALS
\begin{displaymath}
P\als{A}Q = \als{A}' = \left(
\begin{bmatrix}
1 & . & . & . 
\end{bmatrix},
\begin{bmatrix}
x & 1 & . & . \\
1 & y & -1 & -1 \\
0 & 0 & x & 0 \\
. & . & . & 1
\end{bmatrix},
\begin{bmatrix}
. \\ . \\ . \\ 1
\end{bmatrix}
\right)
\end{displaymath}
in which one can eliminate row~3 and column~3
(and ---after an appropriate row operation--- also
the last row and column).

However, minimization can be accomplished much easier:
Firstly we observe that the left subfamily $s_{\block{2:3}}$
(of $\als{A}$) is
$\field{K}$-linearly independent.
Also the right subfamily
$t_{\block{1:2}}$ is $\field{K}$-linearly independent.
For the left family with respect to the first pivot block
we consider the extended ALS
(see also Remark~\ref{rem:mr.extals})
\begin{displaymath}
\hspace{1.68em}\left|\hspace{-2em}%
\begin{bmatrix}
1 & -1 & . & . & . \\\hline\tabstrut
. & x & 1 & . & . \\
. & . & y & -1 & . \\
. & . & -1 & x & -x \\
. & . & . & . & 1
\end{bmatrix}
s =
\begin{bmatrix}
. \\\hline\tabstrut . \\ . \\ . \\ 1
\end{bmatrix}
\right.
\end{displaymath}
of $\als{A}$, where the upper row and the left column are indexed by zero.
Now we add row~3 to row~1,
subtract column~1 from column~3 and
add column~1 to column~4
(this transformation can be found by solving a linear system of equations):
\begin{displaymath}
\hspace{1.68em}\left|\hspace{-2em}%
\begin{bmatrix}
1 & -1 & . & 1 & -1 \\\hline\tabstrut
. & x & 0 & 0 & 0 \\
. & . & y & -1 & . \\
. & . & -1 & x & -x \\
. & . & . & . & 1
\end{bmatrix}
s =
\begin{bmatrix}
. \\\hline\tabstrut 0 \\ . \\ . \\ 1
\end{bmatrix}.
\right.
\end{displaymath}
Now we can remove row~1 and column~1:
\begin{equation}\label{eqn:mr.extals.1}
\hspace{1.68em}\left|\hspace{-2em}%
\begin{bmatrix}
1 & . & 1 & -1 \\\hline\tabstrut
. & y & -1 & . \\
. & -1 & x & -x \\
. & . & . & 1
\end{bmatrix}
s =
\begin{bmatrix}
. \\\hline\tabstrut . \\ . \\ 1
\end{bmatrix}.
\right.
\end{equation}
Before we do the last (right) minimization step,
we transform the extended ALS back into a ``normal''
by exchanging columns~1 and~2, scaling the (new)
column~1 by $-1$ and subtract it from column~3:
\begin{displaymath}
\hspace{1.68em}\left|\hspace{-2em}%
\begin{bmatrix}
1 & -1 & . & 0 \\\hline\tabstrut
. & 1 & y & -1 \\
. & -x &-1 & 0 \\
. & . & . & 1
\end{bmatrix}
s =
\begin{bmatrix}
. \\\hline\tabstrut . \\ . \\ 1
\end{bmatrix}.
\right.
\end{displaymath}
(If necessary right minimization steps can be performed
until one reaches a pivot block with corresponding
non-zero entry in row~0.)
Now we can remove row~0 and column~0 again.
The last step to a minimal ALS for
$fg = (1-yx)\inv$ is trivial:
\begin{displaymath}
\begin{bmatrix}
1 & y & 0 \\
-x &-1 & 0 \\
. & . & 1
\end{bmatrix}
s =
\begin{bmatrix}
1 \\ . \\ 1
\end{bmatrix}.
\end{displaymath}
After removing the last row and column
we exchange the two rows to get a
\emph{standard} ALS.
\end{Example}

Now at least one question should have appeared:
\emph{How} can one prove ---for a given block index $k$---
the $\field{K}$-linear independence of the left
(respectively right) subfamily
$s_{\block{k:m}}$ (respectively $t_{\block{1:k}}$)
in general, assuming that
$s_{\block{k+1:m}}$ (respectively $t_{\block{1:k-1}}$)
is $\field{K}$-linearly independent?
For an answer some preparation is necessary.

\begin{lemma}[\protect{%
\cite[Lemma~1.2]{Cohn1999a}
}]\label{lem:mr.cohn99.12}
Let $f \in \field{F}$ given by the linear representation
$\pi_f = (u,A,v)$ of dimension $n$.
Then $f=0$ if and only if there exist invertible matrices
$P,Q \in \field{K}^{n \times n}$ such that
\begin{displaymath}
P \pi_f Q= \left(
\begin{bmatrix}
\tilde{u}_{\block{1}} & 0
\end{bmatrix},
\begin{bmatrix}
\tilde{A}_{1,1} & 0 \\
\tilde{A}_{2,1} & \tilde{A}_{2,2}
\end{bmatrix},
\begin{bmatrix}
0 \\ \tilde{v}_{\block{2}}
\end{bmatrix}
\right)
\end{displaymath}
for square matrices $\tilde{A}_{1,1}$ and $\tilde{A}_{2,2}$.
\end{lemma}

\begin{theorem}[Left Block Minimization]\label{thr:mr.lmin}
Let $\als{A}=(u,A,v) = (1,A,\lambda)$ be an ALS of dimension~$n$
with $m = \pivot(\als{A}) \ge 2$ pivot blocks of size $n_i = \dim_i(\als{A})$.
Let $k \in \{ 1, 2, \ldots, m-1 \}$ such that
the left subfamily
$s_{\block{k+1:m}}$ with respect to the block decomposition
$\als{A}^{[\block{k}]}$ 
is $\field{K}$-linearly independent while
$s_{\block{k:m}}$ is $\field{K}$-linearly dependent.
Then there exists a block row transformation
$(P,Q) = \bigl( P(\bar{T}, T), Q(\bar{U}, U) \bigr)_{\block{k}}$,
such that $\tilde{\als{A}} = P \als{A} Q$ has the form
\begin{equation}\label{eqn:mr.lmin}
\begin{bmatrix}
A_{1:,1:} & \tilde{A}_{1:,k'} & \tilde{A}_{1:,k''} & \tilde{A}_{1:,:m} \\
. & \tilde{A}_{k',k'} & 0 & 0 \\
. & \tilde{A}_{k'',k'} & \tilde{A}_{k'',k''} & \tilde{A}_{k'',:m} \\
. & . & . & A_{:m,:m}
\end{bmatrix}
\begin{bmatrix}
\tilde{s}_{\block{1:}} \\ 0 \\ \tilde{s}_{\block{k}''} \\ s_{\block{:m}}
\end{bmatrix}
=
\begin{bmatrix}
. \\ . \\ . \\ v_{\block{:m}}
\end{bmatrix}
\end{equation}
If the pivot block $A_{k,k}$ is \emph{refined}, then
there exists a particular block row transformation
$(P,Q) = \bigl(P(T), Q(U)\bigr)\mthstrut_{\block{k}}$,
such that the left block minimization equations
\begin{displaymath}
A_{k,k} U + A_{k,:m} + T A_{:m,:m} = 0^{n_k \times n_{k+1:m}}
\quad\text{and}\quad
v_{\block{k}} + T v_{\block{:m}} = 0^{n_k \times 1}
\end{displaymath}
are fulfilled.
\end{theorem}

\begin{proof}
We refer to the block decomposition
\begin{displaymath}
\als{A}^{[\block{k}]} = \left(
\begin{bmatrix}
u_{\block{1:}} & . & . 
\end{bmatrix},
\begin{bmatrix}
A_{1:,1:} & A_{1:,k} & A_{1:,:m} \\
.       & A_{k,k} & A_{k,:m} \\
.       & . & A_{:m,:m}
\end{bmatrix},
\begin{bmatrix}
. \\ . \\ v_{\block{:m}}
\end{bmatrix}
\right).
\end{displaymath}
Due to the $\field{K}$-linear independence of the
left subfamily $s_{\block{k+1:m}}$ there exists an
invertible matrix $\tilde{Q}$ with blocks
$\bar{U}^\circ \in \field{K}^{n_k \times n_k }$
and $U^\circ \in \field{K}^{n_k \times n_{k+1:m}}$,
such that $(\tilde{Q}\inv s)_{n_{1:k-1}+1} = 0$,
that is, the first component in $s_{\block{k}}$ can be eliminated.
Let
\begin{displaymath}
A' = 
\begin{bmatrix}
A_{k,k} & A_{k,:m} \\
. & A_{:m,:m}
\end{bmatrix}
\begin{bmatrix}
\bar{U}^\circ & U^\circ \\
. & I_{k+1:m}
\end{bmatrix}
\quad\text{and}\quad
v' =
\begin{bmatrix}
. \\ v_{\block{:m}}
\end{bmatrix}.
\end{displaymath}
Then $\als{A}'=(u',A',v')$ is an ALS for $0 \in \field{F}$
and we can apply Lemma~\ref{lem:mr.cohn99.12} to get a transformation
\begin{displaymath}
(P',Q') = \left(
\begin{bmatrix}
\bar{T} & T \\
T_{:m,k} & T_{:m,:m}
\end{bmatrix},
\begin{bmatrix}
\bar{U}' & U' \\
U_{:m,k} & U_{:m,:m}
\end{bmatrix}
\right)
\end{displaymath}
such that $P'A' Q'$ has the respective upper right block of zeros
and ---without loss of generality--- $P'v' = v'$.
Clearly we can choose $U_{:m,:m} = I_{k+1:m}$.
And since $s_{\block{:m}}$ is $\field{K}$-linearly independent,
the block of zeros in
\begin{align*}
P' A' &=
\begin{bmatrix}
\bar{T} & T \\
T_{:m,k} & T_{:m,:m}
\end{bmatrix}
\begin{bmatrix}
A'_{k,k} & A'_{k,:m} \\
. & A_{:m,:m}
\end{bmatrix} \\
&=
\begin{bmatrix}
\bar{T} A'_{k,k} & \bar{T} A'_{k,:m} + T A_{:m,:m} \\
T_{:m,k} A'_{k,k} & T_{:m,k} A'_{k,:m} + T_{:m,:m} A_{:m,:m}
\end{bmatrix}
\end{align*}
is \emph{independent} of $T_{:m,k}$ and $T_{:m,:m}$,
thus we can choose $T_{:m,k}=0$ and $T_{:m,:m}=I_{k+1:m}$.
Now it is obvious that the columns in the lower left block of
\begin{align*}
\!\!\!\!\!\! P'A'Q' &=
\begin{bmatrix}
\bar{T} A'_{k,k} & \bar{T} A'_{k,:m} + T A_{:m,:m} \\
. &  A_{:m,:m}
\end{bmatrix}
\begin{bmatrix}
\bar{U}' & U' \\
U_{:m,k} & I
\end{bmatrix} \\
&=
\begin{bmatrix}
\bar{T} A'_{k,k} \bar{U}' + (\bar{T} A'_{k,:m} + T A_{:m,:m}) U_{:m,k}
  & \bar{T} A'_{k,k} U' + \bar{T} A'_{k,:m} + T A_{:m,:m} \\
A_{:m,:m} U_{:m,k} & A_{:m,:m}
\end{bmatrix}\!
\end{align*}
are linear combinations of the columns of $A_{:m,:m}$
and therefore we can assume that $U_{:m,k}=0$.
Now let $\bar{U}= \bar{U}^\circ \bar{U}'$ and $U = \bar{U}^\circ U' + U^\circ$.
Then $P\als{A}Q$ has 
the desired form~\eqref{eqn:mr.lmin}
for the (for $k>1$ admissible) block row transformation
\begin{displaymath}
(P,Q) = \left(
\begin{bmatrix}
I_{1:k-1} & . & . \\
. & \bar{T} & T \\
. & . & I_{k+1:m}
\end{bmatrix},
\begin{bmatrix}
I_{1:k-1} & . & . \\
. & \bar{U} & U \\
. & . & I_{k+1:m}
\end{bmatrix}
\right).
\end{displaymath}
For the second part we first have to show that \emph{each}
component in $s_{\block{k}}$ can be eliminated by a linear
combination of components of $s_{\block{k+1:m}}$, that is, $n_{k''}=0$.
We assume to the contrary that $n_{k''} > 0$.
But then ---by~\eqref{eqn:mr.lmin}---
$\bar{T} A_{k,k} \bar{U}$ would have
an upper right block of zeros of size $(n_k - n_{k''}) \times n_{k''}$
and therefore (after an appropriate permutation) a lower left,
contradicting the assumption on a refined pivot block.
Hence there exists a matrix
$U \in \field{K}^{n_k \times n_{k+1:m}}$ such that
$s_{\block{k}} - U [s_{\block{k+1}}, \ldots, s_{\block{m}}] = 0$.
By assumption $v_{\block{k}} = 0$.
Now we can apply ---as in Lemma~\ref{lem:mr.min1}---
Lemma~\ref{lem:mr.rt1} with the ALS
$(1,A_{:m,:m},\lambda)$ and $B = -A_{k,k} U - A_{k,:m}$
(and $s_{\block{:m}}$).
Thus there exists a matrix $T \in \field{K}^{n_k \times n_{k+1:m}}$
fulfilling $A_{k,k} U + A_{k,:m} + T A_{:m,:m} = 0$.
Since the last column of $T$ is zero, we have also $T v_{\block{:m}} = 0$.
With $\bar{T} = \bar{U} = I_{1:n_k}$ the transformation $(P,Q)$
is the appropriate particular block row transformation.
\end{proof}

\begin{remark}
For the proof of the second part of the theorem
one can use alternatively Lemma~\ref{lem:mr.cohn99.12}
which is more powerful but with respect to the use of
linear techniques not that obvious.
\end{remark}

\begin{remark}
Notice that the left subfamily $(\tilde{s}_{\block{k}''}, s_{\block{:m}})$
is not necessarily $\field{K}$-linearly independent.
If necessary, one can apply the theorem again
after removing block row and column~$k'$.
\end{remark}

\begin{remark}
For $k=1$, if necessary, one must use an extended ALS,
see Remark~\ref{rem:mr.extals}.
\end{remark}

\begin{Remark}\label{rem:mr.linunabh}
Assuming $\field{K}$-linear independence of the
left subfamily $s_{\block{k+1:m}}$ and a refined
pivot block $A_{k,k}$, the second part of the previous
theorem means nothing less than the possibility
to check $\field{K}$-linear (in-)dependence
of the left subfamily $s_{\block{k:m}}$ by
\emph{linear} techniques!
\end{Remark}

\begin{theorem}[Right Block Minimization]\label{thr:mr.rmin}
Let $\als{A}=(u,A,v) = (1,A,\lambda)$ be an ALS of dimension~$n$
with $m = \pivot(\als{A}) \ge 2$ pivot blocks of size $n_i = \dim_i(\als{A})$.
Let $k \in \{ 2, 3, \ldots, m \}$ such that
the right subfamily $t_{\block{1:k-1}}$
with respect to the block decomposition $\als{A}^{[\block{k}]}$
is $\field{K}$-linearly independent while
$t_{\block{1:k}}$ is $\field{K}$-linearly dependent.
Then there exists a block column transformation
$(P,Q) = \bigl( P(\bar{T}, T), Q(\bar{U}, U)\bigr)\mthstrut^{\block{k}}$,
such that $\tilde{\als{A}} = P \als{A} Q$ has the form
\begin{equation}\label{eqn:mr.rmin}
\begin{bmatrix}
u_{\block{1}} & . & . & .
\end{bmatrix}
=
\begin{bmatrix}
t_{\block{1}} & \tilde{t}_{\block{2'}} & 0 & \tilde{t}_{\block{3}}
\end{bmatrix}
\begin{bmatrix}
A_{1,1} & \tilde{A}_{1,2'} & 0 & \tilde{A}_{1,3} \\
. & \tilde{A}_{2',2'} & 0 & \tilde{A}_{2',3} \\
. & \tilde{A}_{2'',2'} & \tilde{A}_{2'',2''} & \tilde{A}_{2'',3} \\
. & . & . & A_{3,3}
\end{bmatrix}.
\end{equation}
If the pivot block $A_{k,k}$ is refined, then there exists a particular
block column transformation $(P,Q) = \bigl( P(T), Q(U)\bigr)\mthstrut^{\block{k}}$,
such that the right block minimization equations
\begin{displaymath}
A_{1:,1:} U + A_{1:,k} + T A_{k,k} = 0^{n_{1:k-1} \times n_k}
\end{displaymath}
are fulfilled.
\end{theorem}

If one uses alternating left and right block
minimization steps for the minimization,
that is, applying Theorems~\ref{thr:mr.lmin} and~\ref{thr:mr.rmin},
one has to take care that the $\field{K}$-linear independence
of the respective other subfamily is guaranteed.
This is illustrated in the following example.

\begin{Example}\label{ex:mr.alglu}
Let $\als{A} = (u,A,v) = (1,A,\lambda)$ be an ALS with
$m=5$ pivot blocks.
For $k'=2$ we assume that the left subfamily
$s_{\block{k+1:m}}$ is $\field{K}$-linearly independent
and we assume further that there exists a particular
block row transformation $(P,Q)$,
such that the left block minimization equations are
fulfilled, that is, $P\als{A}Q$ has the form
\begin{displaymath}
\begin{bmatrix}
A_{1,1} & A_{1,2} & \tilde{A}_{1,3} & \tilde{A}_{1,4} & \tilde{A}_{1,5} \\
. & A_{2,2} & 0 & 0 & 0 \\
. & . & A_{3,3} & A_{3,4} & A_{3,5} \\
. & . & . & A_{4,4} & A_{4,5} \\
. & . & . & . & A_{5,5}
\end{bmatrix}
\begin{bmatrix}
\tilde{s}_{\block{1}} \\ 0 \\ s_{\block{3}} \\ s_{\block{4}} \\ s_{\block{5}} 
\end{bmatrix}
=
\begin{bmatrix}
. \\ 0 \\ . \\ . \\ v_{\block{5}}
\end{bmatrix}.
\end{displaymath}
If the right subfamily $t_{\block{1:3}}$ is $\field{K}$-linearly
independent, this is \emph{not} necessarily the case for
the right subfamily $t'_{\block{1:3}}$ of the
smaller ALS
\raisebox{0pt}[0pt][0pt]{$\als{A}' = \bigl( P \als{A} Q\bigr)\mthstrut^{[-\block{k}']}$}.
That is, one has to apply Theorem~\ref{thr:mr.rmin} on
$\als{A}'$ with $k=3$ to check that ``again''.
\end{Example}

\begin{algorithm}[Minimizing a refined ALS]\label{alg:mr.minals2}
\ \\
Input: $\als{A} = (u,A,v) = (1,A,\lambda)$ refined ALS (for an element $f$) \\
\rule{2em}{0pt}with $m = \pivot(\als{A}) \ge 2$ pivot blocks of size
  $n_i = \dim_i(\als{A})$ and\\
\rule{2em}{0pt}$\field{K}$-linearly independent subfamilies
  $s_{\block{m}}$ and $t_{\block{1}}$.\\
Output: $\als{A}' = (,,)$, if $f=0$, or\\
\rule{2em}{0pt}a \emph{minimal} refined
  ALS $\als{A}' = (u',A',v')=(1,A',\lambda')$, if $f \neq 0$.

\begin{algtest}
\hbox{}\\[-3.0ex]
\lnum{1:}\>$k := 2$ \\
\lnum{2:}\>while $k \le \pivot(\als{A})$ do \\
\lnum{3:}\>\>$m := \pivot(\als{A})$ \\
\lnum{4:}\>\>$k' := m +1 - k$ \\
\lnum{  }\>\>\textnormal{Is the left subfamily
          \raisebox{0pt}[0pt][0pt]{%
        $(s_{\block{k}'},\overbrace{\mthstrut s_{\block{k'+1}}, \ldots, s_{\block{m}}}^{%
        \text{lin.~independent}})$} $\field{K}$-linearly dependent?}\\
\lnum{5:}\>\>if $\exists\, T,U \in \field{K}^{n_k \times n_{k+1:m}}
    \textnormal{ admissible}: \mathcal{L}_{\block{k}'}(\als{A}) =
      \mathcal{L}_{\block{k}'}(T,U)=0$
    then \\
\lnum{6:}\>\>\>if $k' = 1$ then \\
\lnum{7:}\>\>\>\>return $(,,)$ \\
\lnum{  }\>\>\>endif \\
\lnum{8:}\>\>\>\raisebox{0pt}[0pt][0pt]{%
  $\als{A} := \bigl(P(T) \als{A} Q(U)\bigr)\mthstrut^{[-\block{k}']}$} \\
\lnum{9:}\>\>\>if $k > \max \bigl\{ 2, \frac{m+1}{2} \bigr\}$ then \\
\lnum{10:}\>\>\>\>$k := k-1$ \\
\lnum{   }\>\>\>endif \\
\lnum{11:}\>\>\>continue \\
\lnum{   }\>\>endif \\
\lnum{12:}\>\>if $k' = 1$ and 
  $\exists\, T,U \in \field{K}^{n_k \times n_{k+1:m}}:
    \mathcal{L}_{\block{k}'}(\als{A}^{[+0]})=
    \mathcal{L}_{\block{k}'}(T,U)=0$
  then \\
\lnum{13:}\>\>\>\raisebox{0pt}[0pt][0pt]{%
  $\tilde{\als{A}} := \bigl(P(T) \als{A}^{[+0]} Q(U)\bigr)\mthstrut^{[-\block{k}']}$} \\
\lnum{14:}\>\>\>$\als{A} := \tilde{\als{A}}^{[-0]}$ \\
\lnum{15:}\>\>\>if $k > \max \bigl\{ 2, \frac{m+1}{2} \bigr\}$ then \\
\lnum{16:}\>\>\>\>$k := k-1$ \\
\lnum{   }\>\>\>endif \\
\lnum{17:}\>\>\>continue\\
\lnum{   }\>\>endif \\
\lnum{   }\>\>\textnormal{Is the right subfamily
          \raisebox{0pt}[0pt][0pt]{%
          $(\overbrace{\mthstrut t_{\block{1}}, \ldots, t_{\block{k-1}}}^{%
          \text{lin.~independent}}, t_{\block{k}})$} $\field{K}$-linearly dependent?}\\
\lnum{18:}\>\>if $\exists\, T,U \in \field{K}^{n_{k-1:m} \times n_k}
      \textnormal{ admissible}:
      \mathcal{R}_{\block{k}}(\als{A})= \mathcal{R}_{\block{k}}(T,U)=0$ then \\
\lnum{19:}\>\>\>\raisebox{0pt}[0pt][0pt]{%
  $\als{A} := \bigl(P(T) \als{A} Q(U) \bigr)\mthstrut^{[-\block{k}]}$} \\
\lnum{20:}\>\>\>if $k > \max \bigl\{ 2, \frac{m+1}{2} \bigr\}$ then \\
\lnum{21:}\>\>\>\>$k := k-1$ \\
\lnum{   }\>\>\>endif \\
\lnum{22:}\>\>\>continue \\
\lnum{   }\>\>endif \\
\lnum{23:}\>\>$k := k + 1$ \\
\lnum{   }\>done \\
\lnum{24:}\>return $P\als{A},$ 
        \textnormal{with $P$, such that $Pv = [0,\ldots,0,\lambda']\trp$}
\end{algtest}
\end{algorithm}

\begin{proof}
The admissible linear system $\als{A}$
represents $f=0$ if and only if $s_1 = (A\inv v)_1 = 0$.
Since all systems are equivalent to $\als{A}$,
this case is recognized for $k'=1$ because
by Theorem~\ref{thr:mr.lmin} there is an \emph{admissible}
transformation such that the first left block minimization equation
is fulfilled.
Now assume $f \neq 0$.
We have to show that both the left family $s'$ and
the right family $t'$ of $\als{A}' = (u',A',v')$ are
$\field{K}$-linearly independent respectively.
Let $m' = \pivot(\als{A}')$ and for $k \in \{ 1, 2, \ldots, m' \}$
denote by
\begin{displaymath}
s'_{(k)} = (s'_{\block{m'+1-k}}, s'_{\block{m'+2-k}}, \ldots, s'_{\block{m'}})
\quad\text{and}\quad
t'_{(k)} = (t'_{\block{1}}, t'_{\block{2}}, \ldots, t'_{\block{k}})
\end{displaymath}
the left and the right subfamily respectively.
By assumption $s'_{(1)}$ and $t'_{(1)}$ are $\field{K}$-linearly
independent respectively.
The loop starts with $k=2$.
Only if both $s'_{(k)}$ and $t'_{(k)}$
are $\field{K}$-linearly independent respectively,
$k$ is incremented.
Otherwise a left (Theorem~\ref{thr:mr.lmin}) or a right
(Theorem~\ref{thr:mr.rmin}) minimization step
was successful and the dimension of the current ALS is
strictly smaller than that of the previous.
Hence, since $k$ is bounded from below,
the algorithm stops in a finite number of steps.
(How row~0 and column~0 are
removed from the extended ALS in line~14
is illustrated in Example~\ref{ex:mr.alg}.)
All transformations are such that $\als{A}'$ is a refined
ALS (and therefore in \emph{standard form}).
For $\pivot(\als{A'})=1$ a priori only the left (or the right)
family is $\field{K}$-linearly independent (by assumption).
But if that were not the case for the respective other family,
then the assumption on refined pivot blocks would be contradicted
by Theorem~\ref{thr:mr.rmin} (respectively Theorem~\ref{thr:mr.lmin})
\end{proof}

\begin{remark}
Concerning details with respect to the complexity
of such an algorithm we refer to 
\cite[Remark~33]{Schrempf2017b9}
.
Let $d$ be the number of letters in our alphabet $X$.
For $m = n$ pivot blocks 
and $k < n$ we have $2(k-1)$ unknowns.
By Gaussian elimination one
gets complexity $\complexity(dn^3)$
for solving a linear system for a minimization step,
see \cite[Section~2.3]{Demmel1997a}
.
To build such a system and working on a
linear matrix pencil
$\bigl[\begin{smallmatrix} 0 & u \\ v & A \end{smallmatrix}\bigr]$
with $d+1$ square
coefficient matrices of size $n+1$
(transformations, etc.)
has complexity $\complexity(dn^2)$.
So we get overall (minimization) complexity
$\complexity(d n^4)$.
The algorithm of Cardon and Crochemore \cite{Cardon1980a}
\ has complexity $\complexity(d n^3)$
but works only for regular elements, that is,
rational formal power series.
For $\dim_i(\als{A}) \approx \sqrt{n}$ we get
complexity $\complexity(d n^5)$ and for the
word problem 
\cite{Schrempf2017a9}
\ with $m=2$ we get complexity
$\complexity(d n^6)$.
\end{remark}

\begin{remark}
It is clear that one can adapt the algorithm slightly if
the input ALS is constructed by Proposition~\ref{pro:mr.ratop}
out of two minimal admissible linear systems (for the sum and
the product) in standard form.
\end{remark}

\begin{remark}
The solution of the word problem for two elements
given by \emph{minimal} admissible linear systems is
\emph{independent} of their refinement.
If Algorithm~\ref{alg:mr.minals2} is applied to
an ALS of which it is not known if it is refined,
in some cases it is possible to check if the
ALS $\als{A}'$ is minimal, for example if
$\dim(\als{A}') = \pivot(\als{A}')$.
If the pivot blocks are bigger but the right upper
structure is ``finer'' one can instead ---for $f \neq 0$---
try to minimize the inverse $(\als{A}')\inv$.
In concrete situations there might be other
possibilities to reach minimality.
\end{remark}

\begin{Remark}\label{rem:mr.applications}
Apart from minimization, this algorithm can be used to check if
$f$ is a left factor of an element $fg$ which is relevant for
the minimal factor multiplication 
\cite[Theorem~5.2]{Schrempf2017c9}
. And one can check if two elements are \emph{disjoint}
which is important for the \emph{primary decomposition}
\cite[Theorem~2.3]{Cohn1999a}
.
\end{Remark}

Another aspect of Algorithm~\ref{alg:mr.minals2} (respectively
Theorem~\ref{thr:mr.lmin} and~\ref{thr:mr.rmin}) becomes visible
immediately with Proposition~\ref{pro:mr.cohn94.47} and
Remark~\ref{rem:mr.linunabh}.
The importance of the following theorem becomes clear
if one needs to check $\field{K}$-linear (in-)dependence
of an arbitrary family $(f_1, f_2, \ldots, f_n)$ 
over the free field and it is not possible (anymore)
to take a representation as formal power series
(with coefficients over $\field{K}$).

\begin{theorem}[``Linear'' Characterization of Minimality]
\label{thr:mr.minchar}
A \emph{refined} admissible linear system $\als{A}=(1,A,\lambda)$
for an element in the free field $\field{F}$
with $m=\pivot(\als{A}) \ge 2$ pivot blocks and $\lambda \neq 0$
is minimal if and only if
neither the left block minimization equations
$\mathcal{L}_{\block{k}}(\als{A})$ for $k \in \{ 1, 2, \ldots, m-1 \}$
nor the right block minimization equations
$\mathcal{R}^{\block{k}}(\als{A})$ for $k \in \{ 2,3, \ldots, m \}$
admit a solution.
\end{theorem}

\begin{proof}
From the existence of a solution non-minimality follows immediately
since in this case ---after the appropriate transformation---
rows and columns can be removed.
And for non-minimality Proposition~\ref{pro:mr.cohn94.47}
implies that either the left or the right family is $\field{K}$-linearly
dependent.
Without loss of generality assume that it is the left
$s = (s_{\block{1}}, s_{\block{2}}, \ldots, s_{\block{m}})$
with minimal $k \in \{ 1, 2, \ldots, m-1 \}$ such that
the left subfamily $(s_{\block{k+1}}, \ldots, s_{\block{m}})$
is $\field{K}$-linearly \emph{in}dependent.
Since the pivot blocks are refined, Theorem~\ref{thr:mr.lmin}
implies the existence of a particular block row transformation
$(P,Q)_{\block{k}}$ and therefore a solution of the $k$-th left
block minimization equations.
\end{proof}

\section{Applications}\label{sec:mr.applications}

Since the focus of this work is mainly minimization
and one dedicated to ``minimal'' rational operations
---collecting all techniques for practical application---
is already available
\cite{Schrempf2018c2}
, only two applications are illustrated in the following
examples. For other applications see also
Remark~\ref{rem:mr.applications}.

\begin{Example}[Hua's Identity
\cite{Amitsur1966a}
]\label{ex:mr.hua}
We have:
\begin{equation}\label{eqn:mr.hua1}
x - \bigl(x\inv + (y\inv - x)\inv \bigr)\inv = xyx.
\end{equation}
\end{Example}

\begin{proof}
Minimal admissible linear systems for $y\inv$ and $x$ are
\begin{displaymath}
\begin{bmatrix}
y
\end{bmatrix}
s =
\begin{bmatrix}
1
\end{bmatrix}
\quad\text{and}\quad
\begin{bmatrix}
1 & -x \\
. & 1 
\end{bmatrix}
s =
\begin{bmatrix}
. \\ 1
\end{bmatrix}
\end{displaymath}
respectively.
The ALS for the difference $y\inv -x$,
\begin{displaymath}
\begin{bmatrix}
y & -y & . \\
. & 1 & -x \\
. & . & 1 
\end{bmatrix}
s = 
\begin{bmatrix}
1 \\ . \\ -1
\end{bmatrix}, \quad
s =
\begin{bmatrix}
y\inv - x \\
-x \\
-1
\end{bmatrix}, \quad
t = 
\begin{bmatrix}
y\inv & -1 & y\inv - x
\end{bmatrix}
\end{displaymath}
is minimal because the left family $s$ is $\field{K}$-linearly
independent and the right family $t$ is $\field{K}$-linearly
independent (Proposition~\ref{pro:mr.cohn94.47}).
Clearly we have $1\in R(y\inv -x)$.
Thus, by Lemma~\ref{lem:mr.forR},
there exists an admissible transformation
\begin{displaymath}
(P,Q) = \left(
\begin{bmatrix}
. & 1 & . \\
1 & . & 1 \\
. & . & 1
\end{bmatrix},
\begin{bmatrix}
1 & . & . \\
1 & 1 & . \\
. & . & 1
\end{bmatrix}
\right),
\end{displaymath}
that yields the ALS
\begin{displaymath}
\begin{bmatrix}
1 & 1 & -x \\
. & -y & 1 \\
. & . & 1 
\end{bmatrix}
s =
\begin{bmatrix}
. \\ . \\ -1
\end{bmatrix}.
\end{displaymath}
Now we can apply the inverse of type~$(1,1)$:
\begin{displaymath}
\begin{bmatrix}
1 & y \\
-x & -1
\end{bmatrix}
s =
\begin{bmatrix}
. \\ 1
\end{bmatrix},
\quad
s = 
\begin{bmatrix}
(y\inv -x)\inv \\
-(1-xy)\inv
\end{bmatrix}.
\end{displaymath}
This system represents a regular element
$(y\inv - x)\inv = (1-yx)\inv y$,
and therefore can be transformed into a regular ALS
(Definition~\ref{def:mr.reg})
by scaling row~2 by $-1$.
Then we add $x\inv$ ``from the left'':
\begin{displaymath}
\begin{bmatrix}
x & -x & . \\
. & 1 & y \\
. & x & 1
\end{bmatrix}
s = 
\begin{bmatrix}
1 \\ . \\ -1
\end{bmatrix}, \quad
s =
\begin{bmatrix}
x\inv + (y\inv -x)\inv \\
(y\inv -x)\inv \\
-(1-xy)\inv
\end{bmatrix}.
\end{displaymath}
This system is minimal and ---after adding row~3 to row~1
(to eliminate the non-zero entry in the right hand side)---
we apply the (minimal) inverse of type~$(0,0)$:
\begin{equation}\label{eqn:mr.hua2}
\begin{bmatrix}
-1 & -1 & -x & . \\
. & -y & -1 & . \\
. & -1 & 0 & -x \\
. & . & . & 1
\end{bmatrix}
s =
\begin{bmatrix}
. \\ . \\ . \\ 1
\end{bmatrix}.
\end{equation}
Now we multiply row~1 and the columns~2 and~3 by $-1$
and exchange column~2 and~3 to get the following system:
\begin{displaymath}
\begin{bmatrix}
1 & -x & -1 & . \\
. & 1 & y & . \\
. & . & 1 & -x \\
. & . & . & 1
\end{bmatrix}
s =
\begin{bmatrix}
. \\ . \\ . \\ 1
\end{bmatrix}, \quad
s = 
\begin{bmatrix}
x - xyx \\ -yx \\ x \\ 1 
\end{bmatrix}.
\end{displaymath}
The next step would be a scaling by $-1$ and
the addition of $x$ (by Proposition~\ref{pro:mr.ratop}).
With two minimization steps we would reach again minimality.
Alternatively we can add a \emph{linear} term to a polynomial
(in a polynomial ALS)
---depending on the entry $v_n$ in the right hand side---
directly in the upper right entry of the
system matrix:
\begin{displaymath}
\begin{bmatrix}
1 & -x & -1 & x \\
. & 1 & y & . \\
. & . & 1 & -x \\
. & . & . & 1
\end{bmatrix}
s =
\begin{bmatrix}
. \\ . \\ . \\ 1
\end{bmatrix}, \quad
s = 
\begin{bmatrix}
 - xyx \\ -yx \\ x \\ 1 
\end{bmatrix}.
\end{displaymath}
\end{proof}

\begin{remarks}
The transformation of the ALS~\eqref{eqn:mr.hua2}
is a simple case of the \emph{refinement}
of a pivot block and is discussed in detail
in Section~\ref{sec:mr.standardform}.
Hua's identity is also an example in
\cite{Cohn1994a}
. It is worth to compare both approaches.
\end{remarks}

\begin{Example}[Left GCD]\label{ex:mr.lgcd}

Given two polynomials $p,q \in \freeALG{\field{K}}{X} \setminus \field{K}$,
one can compute the \emph{left} (respectively \emph{right})
\emph{greatest common divisor} of $p$ and $q$
by minimizing an admissible linear system for $p\inv q$ (respectively $p q\inv$).
This is now illustrated in the following example.
Let $p = yx(1-yx)z = yxz - yxyxz$ and
$q = y(1-xy)y = y^2 - yxy^2$.
We want to find $h = \lgcd(p,q)$.
An ALS for $p\inv q$ (constructed out of minimal admissible linear systems
for $p\inv$ and $q$ by Proposition~\ref{pro:mr.mul2}) is
\begin{displaymath}
\begin{bmatrix}
z & 1 & . & . & . & . & . & . & . \\
. & x & -1 & . & . & . & . & . & . \\
. & -1 & y & -1 & . & . & . & . & . \\
. & . & . & x & -1 & . & . & . & . \\
. & . & . & . & y & -y & . & . & . \\
. & . & . & . & . & 1 & -x & 1 & . \\
. & . & . & . & . & . & 1 & -y & . \\
. & . & . & . & . & . & . & 1 & y \\
. & . & . & . & . & . & . & . & 1
\end{bmatrix}
s =
\begin{bmatrix}
. \\ . \\ . \\ . \\ . \\ . \\ . \\ . \\ 1
\end{bmatrix}.
\end{displaymath}
Clearly, this system is \emph{refined}.
How to refine an ALS is discussed in Section~\ref{sec:mr.standardform}.
Note, that there is a close 
connection to the factorization of $p$, for details see
\cite{Schrempf2017b9}
.
As a first step we add column~5 to column~6 and
row~6 to row~4,
\begin{displaymath}
\begin{bmatrix}
z & 1 & . & . & . & 0 & . & . & . \\
. & x &-1 & . & . & 0 & . & . & . \\
. &-1 & y &-1 & . & 0 & . & . & . \\
. & . & . & x &-1 & 0 & -x & 1 & . \\
. & . & . & . & y & 0 & 0 & 0 & 0 \\
. & . & . & . & . & 1 & -x & 1 & . \\
. & . & . & . & . & . & 1 & -y & . \\
. & . & . & . & . & . & . & 1 & y \\
. & . & . & . & . & . & . & . & 1
\end{bmatrix}
s =
\begin{bmatrix}
. \\ . \\ . \\ . \\ . \\ . \\ . \\ . \\ 1
\end{bmatrix},
\end{displaymath}
remove rows and columns 5 and 6,
\begin{equation}\label{eqn:mr.lggt2}
\begin{bmatrix}
z & 1 & . & . & . & . & . \\
. & x & -1 & . & . & . & . \\
. & -1 & y & -1 & . & . & . \\
. & . & . & x &  -x & 1 & . \\
. & . & . & . & 1 & -y & . \\
. & . & . & . & . & 1 & y \\
. & . & . & . & . & . & 1
\end{bmatrix}
s =
\begin{bmatrix}
. \\ . \\ . \\ . \\ . \\ . \\ 1
\end{bmatrix}
\end{equation}
and remember the first (left) divisor $h_1 = y$
we have eliminated. (In the next step with a bigger block
one can see immediately how to ``read'' a divisor
directly from the ALS.)
Now there are two ways to proceed:
If it is not possible to create a zero block in ``L''-form
(like before), one can try to change the upper pivot block
structure to create a ``double-L'' zero block.
Here, this is possible by subtracting column~4 from column~2
and adding row~2 to row~4.
(For details on \emph{similarity unique factorization}
in this context
see \cite{Schrempf2017b9}
.)
Afterwards we apply the (admissible) transformation
\begin{displaymath}
(P,Q) = 
\left(
\begin{bmatrix}
1 & . & . & . & . & . & . \\
. & 1 & . & . & . & 1 & . \\
. & . & 1 & . & 1 & . & . \\
. & . & . & 1 & . & . & . \\
. & . & . & . & 1 & . & . \\
. & . & . & . & . & 1 & . \\
. & . & . & . & . & . & 1
\end{bmatrix},
\begin{bmatrix}
1 & . & . & . & . & . & . \\
. & 1 & . & . & . & . & . \\
. & . & 1 & . & . & 1 & . \\
. & . & . & 1 & 1 & . & . \\
. & . & . & . & 1 & . & . \\
. & . & . & . & . & 1 & . \\
. & . & . & . & . & . & 1
\end{bmatrix}
\right)
\end{displaymath}
to get the ALS
\begin{equation}\label{eqn:mr.lggt3}
\begin{bmatrix}
z & 1 & . & . & 0 & 0 & . \\
. & x & -1 & . & 0 & 0 & y \\
. & . & y & -1 & 0 & 0 & 0 \\
. & . & -1 & x & 0 & 0 & 0 \\
. & . & . & . & 1 & -y & . \\
. & . & . & . & . & 1 & y \\
. & . & . & . & . & . & 1
\end{bmatrix}
s =
\begin{bmatrix}
. \\ . \\ . \\ . \\ . \\ . \\ 1
\end{bmatrix}.
\end{equation}
How to get this transformation $(P,Q)$ is described in principle
in Section~\ref{sec:mr.minimizing}.
One can look directly for a ``double-L'' block transformation.
In the third pivot block of \eqref{eqn:mr.lggt3} one can see
immediately that a further (common) left factor is $h_2 = 1-xy$
because the second equation reads $x s_2 - h_2\inv = 0$.
We have eliminated $(1-xy)\inv (1-xy)$.
Recall that a minimal ALS for $h_2$ is
\begin{displaymath}
\begin{bmatrix}
1 & -x & 1 \\
. & 1 & -y \\
. & . & 1
\end{bmatrix}
s =
\begin{bmatrix}
. \\ . \\ -1
\end{bmatrix},
\end{displaymath}
hence $\rank(h_2) = 3$ and (by Theorem~\ref{thr:mr.mininv})
$\rank(h_2\inv) = 2$. Or, more general, for a (left)
factor $h_i$ with $\rank(h_i) = n_i \ge 2$ we can
construct (by Proposition~\ref{pro:mr.mul2})
an ALS of dimension $2(n_i-1)$. 
After removing rows and columns~$\{ 3,4,5,6 \}$
in the ALS~\eqref{eqn:mr.lggt3}, we obtain
for $p\inv q$ the (in this case \emph{minimal}) ALS
\begin{equation}\label{eqn:mr.lggt4}
\begin{bmatrix}
z & 1 & . \\
. & x & y \\
. & . & 1
\end{bmatrix}
s =
\begin{bmatrix}
. \\ . \\ 1
\end{bmatrix}.
\end{equation}
Hence $h = h_1 h_2 = y(1-xy) = \lgcd(p,q)$.
The second possibility ---starting from ALS~\eqref{eqn:mr.lggt2}---
is to do a right minimization step with respect to
column~5, then one left with respect to rows~2 and~3 and
finally a right (minimization step).
Again one obtains the ALS~\eqref{eqn:mr.lggt4}
(up to admissible scaling of rows and columns)
where the right factor $y$ of $q$ remains.
Therefore $\lgcd(p,q) = y - yxy$.
For further details concerning the minimal polynomial
multiplication (Proposition~\ref{pro:mr.minmul})
we refer to 
\cite{Schrempf2017b9}
.
\end{Example}

\begin{remark}
It can happen that ---after there is no more ``L''-minimization step
possible--- the ALS is \emph{not} minimal,
that is, an additional ``single'' left or right minimization step
can be carried out. More details on that are part of the
general factorization theory 
\cite{Schrempf2017c9}
.
Here it suffices to take a closer look at the ALS~\eqref{eqn:mr.lggt4}:
\emph{both} right factors of $p$ (here $xz$) respectively $q$ (here $y$)
can still be ``read'' directly.
(This would not be possible any more if one left or right step
would be carried out.)
\end{remark}

\begin{Remark}\label{rem:mr.linref}
For a \emph{refined} ALS for $p\inv$, an ALS of dimension~$n$
for $p\inv q$ with ``factors'' of rank $\sqrt{n}$ and an alphabet
with $d$ letters, the complexity for computing the left (or right)
gcd is roughly $\complexity(dn^5)$.
Although in general refinement is difficult because of the necessity
to solve systems of polynomial equations (over a not necessarily
algebraically closed field), especially for polynomials
\emph{linear} techniques are very useful for the factorization
(and therefore for the refinement of the inverse).
As an example we take the polynomial $p = (1-xy)(2+yx)(3-yz)(2-zy)(1-xz)(3+zx)x$
of rank $n=14$ which has already 64~terms. To get the first left
(irreducible) factor (of rank~3) we just need to create an upper right
block of zeros (in the system matrix) of size $2 \times 11$
which can be accomplished
by either using the columns~2--3 and rows~4--13
\emph{or} column~2 and rows~3--13 (for elimination)
\cite{Schrempf2017b9}
.
Both cases result in a \emph{linear} system of equations
because column and row transformations do not ``overlap''.
Solving one of these $2(n-2)$ systems has (at most) complexity
$\complexity(d n^6)$, hence in total we have $\complexity(d n^7)$.
Checking irreducibility of polynomials (using Gr\"obner bases)
works practically up to rank~12 
\cite{Janko2018a}
.
\end{Remark}

\section*{Epilogue}

This work is the last in a series for the development of
tools for the work with linear representations (for elements
in the free field) especially for the implementation in
computer algebra systems.
A ``practical'' guide giving an overview
and an introduction is
\cite{Schrempf2018b}
\ (in German, with remarks on the implementation) and 
\cite{Schrempf2018c2}
.

The main idea is as simple as in the usage of ``classical''
fractions (for elements in the field of rational numbers):
\emph{calculating}, \emph{factorizing} and \emph{minimizing}
(or \emph{cancelling}), for example
\begin{align*}
&\frac{2}{3} \cdot \frac{3}{4}
  = \frac{6}{12} = \frac{2\cdot 3}{2\cdot 2 \cdot 3}
  = \frac{1}{2}\quad\text{or}\\
&\frac{1}{2} + \frac{3}{2}
  = \frac{4}{2}
  = \frac{2\cdot 2}{2} = 2.
\end{align*}
At some point one stops this loop and uses the fraction
(with \emph{coprime} numerator and denominator). 
Clearly, one could simplify things by remembering the
factorization of the numerator (for the product)
and the denominator (for the sum and the product).
In our case of the free field, linear representations
(or admissible linear systems) are just ``free fractions''
\ldots

However, to understand how the transition from using
nc rational expressions (to represent elements in the
free field) to (minimal) admissible linear systems
in \emph{standard form} affects the capabilities of
thinking about (free) nc algebra, one needs to go to the
meta level
\cite{Kraemer2014a}
.

\section*{Acknowledgement}

I am very grateful to Wolfgang Tutschke and
Thomas Hirschler for encouraging me in
following my mathematical way in a difficult time.
And I thank the anonymous referees for the constructive remarks,
in particular to stress the various connections to related
areas and improve the exposition.
The open access funding is provided by University of Vienna.

\ifJOURNAL
\else
\addcontentsline{toc}{section}{Bibliography}
\bibliographystyle{alpha}
\fi
\bibliography{doku}

\end{document}